\documentclass{article}

\usepackage{amsmath}
\usepackage[english]{babel}
\usepackage{amsfonts}
\usepackage{amssymb}
\usepackage{amsthm}
\usepackage[pdftex]{graphicx}
\usepackage{authblk}
\usepackage[numbers]{natbib}

\numberwithin{equation}{section}
\newtheorem{theorem}{\textbf{Theorem}}[section]

\newtheorem{lemma}[theorem]{Lemma}
\newtheorem{remark}[theorem]{\textbf{Remark}}

\title{The Discrete Unbounded Coagulation-Fragmentation Equation with Growth, Decay and Sedimentation}
\author[1]{J. Banasiak}
\author[2]{L.O. Joel}
\author[3]{S. Shindin}
 \affil[1]{Department of Mathematics and Applied Mathematics\\
 	University of Pretoria\\
 	Pretoria, South Africa \& Institute of Mathematics, Technical University of \L\'{o}d\'{z}, \L\'{o}d\'{z}, Poland\\ \textit{e-mail: jacek.banasiak@up.ac.za}}\date{}
\affil[2,3]{School of Mathematics, Statistics and Computer Science\\
           University of Kwazulu-Natal\\
           Westville Campus, Durban\\
           South Africa\\ \textit{e-mails: oluwaseyejoel@gmail.com \& shindins@ukzn.ac.za}}\date{}

\begin{document}

\maketitle

\pagestyle{myheadings}
\markboth{J. Banasiak, L.O. Joel and S. Shindin}{Discrete Coagulation-Fragmentation with Decay}
\thispagestyle{empty}

\begin{abstract}
\noindent
In this paper we study the discrete coagulation--fragmentation models with growth, decay
and sedimentation. We demonstrate the existence and uniqueness of classical global
solutions provided the linear processes are sufficiently strong. This paper extends several previous results both by considering a more general model and and also signnificantly weakening the assumptions. Theoretical conclusions are
supported by numerical simulations.
\end{abstract}


\section{Introduction}
Coagulation refers to the aggregation of smaller clusters of particles to form larger ones. Some terms that
are being used interchangeably with coagulation are aggregation and clustering. The first mathematical model to study such processes was proposed by Marian von Smoluchowski, who in
\cite{Smoluchowski1916, Smoluchowski1917} introduced and analysed the following system of equations
\begin{equation}\label{eq1.1}
\frac{d f_{i}(t)}{d t} =  \frac{1}{2} \sum_{j=1}^{i-1} k_{i-j,j}f_{i-j}(t)f_{j}(t)
-  \sum_{j=1}^{\infty} k_{i,j} f_{i}(t) f_{j}(t),\quad i\ge 1,\; t>0.
\end{equation}
The system describes the so-called discrete coagulation, where it is assumed that any cluster consists of a finite number of monomers; that is, building blocks of minimal size (or, interchangeably, mass), taken to be equal to 1. The number of clusters of size $i$, called $i$-clusters, at any time $t\geq 0$ is given by $f_i(t)$ and   $k_{i,j}, i,j\geq 1 $ are the coagulation rates, i.e. the rates at which  the clusters of mass $i$ and $j$ join each other to form a cluster of mass $i+j.$ The first term on the right-hand side, called the gain term,  describes the rate of the emergence of
$i$-clusters by coagulation of $j$ and $j-i$ clusters with $j<i$, while the second term, called the loss term, gives the rate of removal  of $i$-clusters due to the
coalescence with other ones. The factor  $\frac{1}{2}$ ensures that double counting due to symmetry is avoided.

The Smoluchowski equations describe an irreversible process. In reality coagulation is almost always coupled with a fragmentation process in which clusters split into smaller ones. The first model including fragmentation is due to Becker and D\"oring \cite{Becker1935} who, however, only considered the situation in which a monomer could join or leave a cluster according to the scheme
\begin{align*}
&C_r + C_1   \rightarrow    C_{r+1},   \qquad  \text{coagulation process}, \\
&C_{r+1}      \rightarrow    C_r + C_1,  \qquad  \text{fragmentation process}.
\end{align*}
A comprehensive analysis of the Becker--D\"{o}ring model can be found in Wattis, \cite{Wattis2006}.  As far as the coagulation is concerned, the Becker--D\"{o}ring model is a simplification of the Smoluchowski equation and, alleviating this shortcoming, Blatz and Tobolsky formulated a full fragmentation--coagulation model in \cite{blatz1945}. As noted in \cite{Collet2004}, the Becker-D\"oring model can describe an early stage of the fragmentation--coagulation process, called the \textit{nucleation stage},  when the monomers interact to build bigger clusters but still make up
the majority of the ensemble.

We note that there is a parallel continuous theory of fragmentation--co\-agu\-lat\-ion processes in which it is assumed that the size of a particle can be any positive number. We are not concerned with such models here and the interested reader is referred to the recent monograph, \cite{BLL2018}.

Further research on fragmentation--coagulation models have led to extensions that include other internal or external processes such as diffusion or transport of clusters in space, or their decay or growth, see e.g. \cite{Canizo2010, Collet1996, Wrzo97}. In particular, in applications to life sciences the clusters consist of living organisms and can change their size not only due to the coalescence or splitting, but also due to internal demographic processes  such as death or birth of organisms inside, see \cite{Oku86, Oku, Gue, Jack90, Mir18}. Also, in some fields, notably in the phytoplankton dynamics, the removal of whole clusters due to their sedimentation is an important process that is responsible for rapid clearance of the organic material from the surface of the sea.  The removal of clusters of suspended solid particles from a mixture is also important in water treatment, biofuel production, or beer fermentation. In all these applications the size distribution of the clusters is a crucial parameter controlling the efficacy of the process, \cite{Ackleh1, Jack90,Mir18}. Thus, models coupling the fragmentation, coagulation, birth, death and removal processes are relevant in many applications and hence in  this paper we focus on analysing the following  comprehensive system,
\begin{equation}\label{eq1.3}
\begin{split}
\frac{d f_{i}}{d t} &= g_{i-1}f_{i-1} - g_{i}f_{i} + d_{i+1}f_{i+1} - d_{i}f_{i} - s_{i}f_{i} -  a_{i}f_{i}  \\
& + \sum_{j= i+1}^{\infty} a_{j} b_{i,j} f_{j} + \frac{1}{2}\sum_{j=1}^{i-1}k_{i-j,j}
f_{i-j}f_{j} -  \sum_{j=1}^{\infty} k_{i,j} f_{i} f_{j}, \\
&f_{i}(0) = \mathring f_{i}, \quad i \geq 1,
\end{split}
\end{equation}
where $f = (f_i)_{i=1}^\infty$ gives the numbers $f_i$ of clusters of mass $i$, and, to shorten notation, we adopted the convention that $g_0 =f_0 =0$. The nonnegative
coefficients $g_i$, $d_{i}$ and $s_{i},$ $i\geq 1$, control the growth, the decay and the sedimentation processes, respectively.
The fragmentation rates are given by $a_{i}, $ while $ b_{i,j} $ is the average number of $i$-mers produced
after the breakup of a $j$-mer, with $ j \ge i$. The difference operators $f \to (g_{i-1}f_{i-1} - g_{i}f_{i})_{i=1}^\infty$
and $f \to (d_{i+1}f_{i+1} - d_{i}f_{i})_{i=1}^\infty$ describe the rate of change of the number of particles due
to, respectively, the birth and death/decay process. The form of these operators can be obtained as in the standard birth-and-death Markov process, e.g. \cite{BR}, assuming that only one birth or death event can occur in a cluster of cells in a short period of time so that an  $i$-cluster only may become an $i+1$, or an $i+1$-cluster.
If we set $g_{i} = d_{i} = s_{i} = 0, i \geq 1, $ then we arrive at the classical mass-conserving coagulation-fragmentation equation.

Since clusters can only fragment into smaller pieces, we have
\[
 a_{1} = 0, \qquad  b_{i, j} = 0,  \quad  i \geq  j.
\]
We also assume that all clusters that are not monomers undergo fragmentation;
that is,  $a_{i} > 0$ for $i \ge 2$. Since the fragmentation process only consists
in the rearrangement of the total mass into clusters, it must be conservative
and hence we require
\[
\sum_{i=1}^{j-1} i b_{i,j} = j, \quad j \geq 2.
\]
The main aim of this paper is to prove the existence of global classical solutions to \eqref{eq1.3} and provide a working numerical scheme for solving it. Thanks to recent results showing that the linear part of the problem generates an analytic semigroup, \cite{Banasiak2018}, in this paper we significantly extended  well-posedness results existing in the literature, see e.g.  \cite{Bana12b, daCo95}, by considering  more general models, removing many constraints on the coefficients of the problem and proving all results for classical solutions, and for weak solutions considered by most earlier works.

The paper is organized as follows. In Section \ref{secPrem} we recall the main results of \cite{Banasiak2018} concerning the analysis of the linear part of the problem and introduce relevant tools from the interpolation theory. Section \ref{secGlo} contains the proof of the global well-posedness of the problem. The idea of the analysis is classical but due to numerous technicalities specific to the problem at hand, as well as because some interim estimates are used in Section \ref{secNum}, we decided to provide an outline of the proofs.  Finally, in Section \ref{secNum} we construct finite dimensional truncations of \eqref{eq1.3}, prove the convergence of their solutions to the solutions of \eqref{eq1.3} as the dimension of the truncation goes to infinity, and use the obtained results to provide rigorous numerical simulations.

\textbf{Acknowledgements.} The research was supported by the NRF grants N00317 and N102275, and  the National Science Centre, Poland, grant 2017/25/B /ST1/00051.

\section{Preliminaries}\label{secPrem}
\subsection{The linear part}
The linear part of the model \eqref{eq1.3} is discussed in details in \cite{Banasiak2018}. In what follows
we briefly mention the key results obtained there that are pertinent to the analysis
of the complete nonlinear model.

In the space
\[
X_p = \Bigl\{f:=(f_i)_{i = 1}^\infty\, : \, \|f\|_p = \sum_{i\ge 1} i^p |f_i|\Bigr\},
\]
we consider the operators $(T_p, D(T_p))$, $(G_p, D(G_p))$, $(D_p, D(D_p))$ and $(B_p, D(B_p))$
defined by
\begin{equation*}
\begin{aligned}
& [ T_p f]_i = - \theta_i f_i, \qquad  \qquad
[ G_p  f]_i =   g_{i-1}f_{i-1},  \\
& [ D_p  f]_i =   d_{i+1} f_{i+1},  \qquad
[B_p  f]_{i} = \sum_{j= i+1}^{\infty} a_{j} b_{i,j} f_{j}, \quad i \geq 1,
\end{aligned}
\end{equation*}
where $ \theta_i := a_i + g_i + d_i + s_i, i\geq 1,$ and  $g_0 = a_1=0$. Further, we denote \begin{equation}
\Delta_i^{(p)} :=i^p-\sum\limits_{j=1}^{i-1}j^pb_{j,i}, \quad i\geq 2,\; p\geq 0.\label{Delta}
\end{equation}
Then the following holds (see \cite{Banasiak2018} for the details):
\begin{theorem}\label{thm2.1}

If for some $p_0>1$
\begin{equation}\label{eq2.1}
\liminf\limits_{i\to \infty}\frac{a_i}{\theta_i}\frac{\Delta_i^{({p_0})}}{i^{p_0}}>0,
\end{equation}
then for any $p>1$ the sum $({Y}_p, D({Y}_p)) = (T_p+G_p+D_p+B_p, D(T_p))$
generates a positive analytic $C_0$-semigroup $\{S_p(t)\}_{t\ge 0}$ in $X_p$.
\end{theorem}
\begin{proof}
This result for $p\geq p_0$ was proved in \cite[Theorem 2]{Banasiak2018}. Here we show that if \eqref{eq2.1} holds for some $p_0>1$, then it also holds for all $p\in (1,p_0]$  and thus the argument of the proof of  \cite[Theorem 2]{Banasiak2018} applies for all $p>1$. We let $\phi_i(p) = \frac{\Delta_i^{({p})}}{i^{p}}$, $i\ge 2$, $p>1$. It is easy to verify that for any  $i\ge 2$ and $p>1$
$0< \phi_i(p)< 1$. This indicates, in particular, that condition
\eqref{eq2.1} is equivalent to the existence of constants $\alpha>0$ and $\beta>0$ such that
$$
\inf_{i\geq 2} \frac{a_i}{\theta_i} = \alpha >0
$$
and
\begin{equation}
\inf_{i\geq 2} \phi_i(p_0) = \beta >0.
\label{betadef}
\end{equation}
Straightforward computations  yield for $p>1, i\geq 2,$
\[
\phi_i'(p) = \frac{1}{i^p}\sum_{j=1}^{i-1} \ln\Bigl(\frac{i}{j}\Bigr) j^p b_{j,i}> 0, \quad
\phi_i''(p) = -\frac{1}{i^p}\sum_{j=1}^{i-1} \ln^2\Bigl(\frac{i}{j}\Bigr) j^p b_{j,i}< 0,
\]
so that each quantity $\phi_i(p)$ is monotone increasing and strictly concave on $(1,\infty)$. The monotonicity ensures that if \eqref{betadef} is satisfied for some $p_0$, then it is satisfied for and $p>p_0$. On the other hand, since $\phi_i(1)=0$, $i\ge 2$, the concavity
implies that for $p \in (1,p_0]$
$$
\phi_i(p) \geq  \frac{f(p_0)}{p_0-1}(p-1) \geq \beta \frac{p-1}{p_0-1}
$$ and hence
\[
\phi_i(p) \ge  \min\left\{1,\frac{p-1}{p_0-1}\right\}\beta>0,\quad
p>1.
\]
Hence \eqref{eq2.1} holds for all $p>1$ and, as in \cite[Theorem 2]{Banasiak2018}, we conclude that for each $p>1$,
the sum $({Y}_p, D({Y}_p)) = (T_p+G_p+D_p+B_p, D(T_p))$
generates a positive analytic $C_0$-semigroup $\{S_p(t)\}_{t\ge 0}$ in $X_p$.
\end{proof}

We mention that if the sedimentation is sufficiently strong,  the generation result extends to $p=1$. Indeed, in the same way as in \cite{Banasiak2017, Banasiak2018} one can show that the assumption
\begin{equation}\label{eq2.2}
\liminf_{i\to \infty}\left(s_i + \frac{d_i-g_i}{i}\right)\frac{1}{\theta_i} > 0,
\end{equation}
ensures that the sum $({Y}_1, D({Y}_1)) = (T_1+D_1+B_1+G_1, D(T_1))$ generates a  positive quasi-contractive analytic
$C_0$-semigroup $\{S_1(t)\}_{t\ge 0}$ in $X_1$.

\subsection{Intermediate Spaces}
\label{sec2.3}
In view of Theorem~\ref{thm2.1}, we define
$$X_{p,1} = \{f\, :\, f \in X_{p} \cap D(T_p), \| f \|_{p, 1} = \| (1+\theta)f  \|_p\},$$
where $\theta := (\theta_i)_{i=1}^\infty,$ and consider
the intermediate spaces $X_{p, \alpha} = (X_{p}, X_{p, 1})_{\alpha,1}$, $0<\alpha<1$, where
$(\cdot,\cdot)_{\alpha,1}$ is the standard real interpolation functor (see \cite[Theorem 5.4.1]{bergh1976}).
We observe that both spaces $X_p$ and $X_{p,1}$ are weighted versions of $\ell^1$,
consequently, the weighted Stein-Weiss interpolation theorem \cite[Theorem 5.4.1]{bergh1976} applies
and the norm in $X_{p, \alpha} $ is given by the expression
\begin{equation}\label{eq2.3}
\|f\|_{p, \alpha} = \sum_{i\ge 1} i^p (1+\theta_i)^\alpha |f_i|,\quad 0<\alpha<1.
\end{equation}
The interpolation functor $(\cdot,\cdot)_{\alpha,1}$ is known to be exact \cite{bergh1976}. Hence,
for any bounded linear operator $T \in L(X_p, X_p) \cap L(X_p, X_{p,1})$, we have
$T\in L(X_p, X_{p, \alpha})$ and $ \|T\|_{X_p \to X_{p, \alpha}}
= \|T\|_{X_p \to X_p}^{1-\alpha}\| T \|_{X_p \to X_{p, 1}}^\alpha$, $0 < \alpha < 1$.
In our case we observe  that $\|S_p(t)\|_{X_p \to X_p} \le c_{0,p}e^{\omega_p t}$,
for all $t\ge 0$ and some fixed $c_{0,p},\omega_p>0$,
while, due to the analyticity,  $\|S_p(t)\|_{X_p \to X_{p,1}} \le \tfrac{c_{1,p}}{t}e^{\omega_p t}$, $t>0$,
see \cite[Theorem II.4.6(c)]{Engel2000}. It follows that
$S_p(t)\in L(X_p, X_{p, \alpha})$ and
\begin{subequations}\label{eq2.4}
\begin{equation}\label{eq2.4a}
\| S_p (t) \|_{X_p \to X_{p, \alpha}}  \le  \frac{c_{0,p}^{1-\alpha} c_{1,p}^\alpha}{t^\alpha}
e^{\omega_p t} =:\frac{c_{\alpha,p}}{t^\alpha}e^{\omega_p t}, \quad t>0, \quad 0 < \alpha < 1.
\end{equation}
In addition, since $S_p(t)\in L(X_{p, 1}, X_{p, 1})\cap L(X_p, X_p)$, see
\cite[Theorem II.4.6(c)]{Engel2000}, similar arguments imply
\begin{equation}\label{eq2.4b}
\| S_p (t) \|_{X_{p, \alpha} \to X_{p, \alpha}}  \le  c_{0,p}e^{\omega_p t}, \quad t\ge0, \quad 0 < \alpha < 1
\end{equation}
and
\begin{equation}\label{eq2.4c}
\| S_p(t) \|_{X_{p, \alpha} \to X_{p,1}}  \le  \frac{c_{0,p}^\alpha c_{1,p}^{1-\alpha}}{t^{1-\alpha}}
e^{\omega_p t} =: \frac{c_{\alpha,p}'}{t^{1-\alpha}}e^{\omega_p t},\quad t>0, \quad 0 < \alpha < 1.
\end{equation}
\end{subequations}
In the sequel, we make use of the operator
\[
(Y_{\gamma,p,\beta}, D(T_p)):=(Y_p + \gamma T_{p,\beta}, D(T_p)),
\quad [T_{p,\beta}f]_i = -(1+\theta_i)^\beta f_i,\quad i\ge 1,
\]
where $\gamma$ is a positive parameter and $0\le \beta \leq 1$. Using \cite[Corollary 3.2.4]{Pa} for $0\le \beta <1$ (and obvious addition if $\beta =1$)  and an argument analogous to that in the proof of \cite[Theorem 5.1]{Bana12b}, we verify that under assumptions of Theorem~\ref{thm2.1}, $(Y_{\gamma,p,\beta}, D(T_p))$
generates a positive analytic $C_0$-semigroup  $\{S_{\gamma,p,\beta}(t)\}_{t\ge 0}$ in $X_p$ for all
$p>1$ and $\gamma>0$. Furthermore,
$\|S_{\gamma,p,\beta}(t)\|_{X_p\to X_p}\le \|S_{p}(t)\|_{X_p\to X_p}$,
$\|S_{\gamma,p,\beta}(t)\|_{X_p\to X_{p,\alpha}}\le \|S_{p}(t)\|_{X_p\to X_{p,\alpha}}$ and
$\|S_{\gamma,p,\beta}(t)\|_{X_p\to X_{p,1}}\le \|S_{p}(t)\|_{X_p\to X_{p,1}}$, uniformly in
$\gamma>0$ and $0\le \beta\le 1$,\footnote{For non-negative sequences $(f_i)_{i\ge1}$ with
finitely  many nonzero entries, the respective bounds are easy consequences of the positivity of
operators $\{S_p(t)\}_{t\ge 0}$,
$\{S_{\gamma,p,\beta}(t)\}_{t\ge 0}$, $-T_{p,\beta}$ and the variation of constant formula.
General result follows immediately from the standard monotone limit argument.}
so that the estimates \eqref{eq2.4}, with the constants $c_{0,p}$, $c_{1,p}$, $c_{\alpha,p}$ and
$c_{\alpha,p}'$,
hold for the operator $S_{\gamma,p,\beta}(t)$ as well. In fact, $\{S_{\gamma,p,\beta}(t)\}_{t\ge 0}$
is substochastic when $\gamma>0$ is sufficiently large, i.e. \eqref{eq2.4} hold with $\omega_p=0$ in that case.

\section{Global well-posedness}\label{secGlo}
In this section, we provide a well-posedness analysis of the complete semilinear model \eqref{eq1.3}.
We assume that all the conditions of Theorem~\ref{thm2.1} are satisfied. In addition, we impose the
following bound on the coefficients of the coagulation kernel
\begin{equation}\label{eq3.1}
k_{i,j} \le \kappa ((1+ \theta_{i})^{\alpha} + (1+\theta_{j})^{\alpha}),\quad i,j\ge 1,\; 0<\alpha<1.
\end{equation}
The analysis proceeds in a number of simple but technical steps. For the readers convenience the proofs
of main results are broken into a sequence of short independent statements.

\subsection{Local analysis}
The analysis presented below is fairly standard. We convert \eqref{eq1.3} into an equivalent Volterra
type integral equation and then employ a variant of the classical Picard-Lindel\"of iterations to obtain
local mild solutions. Then, with some additional work it is not difficult to verify that the mild solutions
are in fact classical. The calculations are similar to that of e.g. \cite[Section 6.3]{Pa} or \cite{Bana12b} but, as some intermediate estimates are needed for calculations in Section \ref{secNum}, we provide an outline of the proofs.

\begin{lemma}\label{lm3.1}
Assume for some $p>1$ conditions \eqref{eq2.1} and \eqref{eq3.1} are satisfied. Then for each
$f_0 \in X_{p,\alpha}^+$ \footnote{Here and in what follows, for a subset $U$ of any of the sequence space considered in the paper, by $U^+$ we denote the subset consisting of all nonnegative sequences in $U$.}
and some $T > 0$, the initial value problem \eqref{eq1.3}
has a unique non-negative mild solution $f \in C([0, T], X_{p,\alpha})$.
\end{lemma}
\begin{proof}
(a) To begin, we cast the equation \eqref{eq1.3} in the form of the Abstract Cauchy Problem (ACP), i.e.
\[
\frac{df}{dt} = Y_{\gamma,p,\alpha} f + F_{\gamma,\alpha}(f),\quad f(0) = f_0\in X_{p,\alpha},
\]
where
\begin{align}
[F_{\gamma,\alpha}(f)]_i &:= \gamma(1+\theta_i)^\alpha f_{i}  + [F_1(f)]_i  - [F_2(f)]_i \nonumber\\&:= \gamma(1+\theta_i)^\alpha f_i+\frac{1}{2}\sum_{j=1}^{i-1} k_{i-j,j} f_{i-j}f_j
- \sum_{j=1}^\infty k_{i,j}f_i f_j,\quad i\ge 1,\label{F1F2}
\end{align}
and $\gamma = (1+\omega_p+2\kappa)(1+c_{0,p}\|f_0\|_{p,\alpha})$.
As noted above, $\{S_{\gamma,p,\alpha}(t)\}_{t\ge 0}$ is substochastic in $X_p$ and
for all $t\in[0, T]$ and some fixed $T>0$, classical solutions of \eqref{eq1.3} satisfy
\begin{equation}\label{eq3.2}
f(t) = S_{\gamma,p,\alpha}(t) f_0 +  \int_{0}^{t} S_{\gamma,p,\alpha} (t - \tau) F_{\gamma,\alpha} (f(\tau)) d\tau.
\end{equation}
We demonstrate that the integral equation \eqref{eq3.2} is locally solvable.

(b) The map $F_{\gamma,\alpha}:X_{p,\alpha} \to X_p$ is bounded and locally Lipschitz continuous provided  \eqref{eq3.1} holds. The argument here is the same as in the proof of \cite[Theorem 5.1]{Bana12b}, leading to
\begin{align}
\nonumber
\|F_{\gamma,\alpha}(f)\|_p &\le \gamma\|f\|_{p,\alpha}+
\sum_{j=1}^{\infty} |f_{j}| \sum_{i=1}^{\infty} i^p k_{i,j} |f_{i}|
+ \frac{1}{2} \sum_{j=1}^{\infty}  |f_j| \sum_{i=1}^{j-1} (i+j)^p k_{j,i} |f_{i}|\\
\label{eq3.3}
&\le \bigl(\gamma+2^{p+1}\kappa\|f\|_{p,\alpha}\bigr)\|f\|_{p,\alpha}
\end{align}
and
\begin{align}
\nonumber
\|F_{\gamma,\alpha}(f) - F_{\gamma,\alpha}(g)\|_p & \le \gamma\|f-g\|_{p,\alpha}
+\sum_{j=1}^{\infty}|f_{j} - g_{j}| \sum_{i=1}^{\infty} i^p k_{i,j} |f_i| \\
\label{eq3.4}
& \le \bigl(\gamma+2^{p+1}\kappa ( \|f\|_{p,\alpha} + \|g\|_{p,\alpha} ) \bigr) \|f-g\|_{p,\alpha}.
\end{align}
We use estimates \eqref{eq3.3} and \eqref{eq3.4} to show that the nonlinear map
\[
M(f) = S_{\gamma,p,\alpha}(t)f_0 + \int_0^t S_{\gamma,p,\alpha}(t-\tau) F_{\gamma,\alpha}(f(\tau)) d\tau,
\]
is a contraction in the closed ball $B_r(f^0) = \{f\, :\, \|f-f^0\|_{C([0,T], X_{p,\alpha})}\leq r\}$ with $0<r<1$ and
\begin{equation}\label{eq3.5}
0<T\le \Bigl(\dfrac{r(1-\alpha)}{2 (1+\omega_p+2^{p+3}\kappa)
c_{\alpha,p} (1+c_{0,p}\|f_0\|_\alpha)^2}\Bigr)^{\frac{1}{1-\alpha}}.
\end{equation}

(c) Let $f^0(t) = S_{\gamma,p,\alpha}(t)f_0$, $t\in [0, T].$ We show that $B_r(f^0)$ is invariant
under the action of $M$. Indeed, for any $f\in B_r(f^0)$
\begin{align*}
&\|f^0 - M(f)\|_{C([0,T], X_{p,\alpha})} \le \max_{0\le t\le T}\int_0^t
\|S_{\gamma,p,\alpha}(t-\tau)F_{\gamma,\alpha}(f(\tau))\|_{p,\alpha} d\tau\\
&\qquad\qquad= \frac{c_{\alpha,p} T^{1-\alpha}}{1-\alpha} \bigl[\gamma + 2^{p+1}\kappa
\|f\|_{C([0,T], X_{p,\alpha})}\bigr] \|f\|_{C([0,T], X_{p,\alpha})}\\
&\qquad\qquad\leq \frac{c_{\alpha,p} T^{1-\alpha}}{1-\alpha} (1+\omega_p+2^{p+2}\kappa)
(1+c_{0,p}\|f_0\|_{p,\alpha})^2\leq r,
\end{align*}
where we used the elementary inequality
\[
\|f\|_{C([0,T], X_{p,\alpha})} \le r + \|f^0\|_{C([0,T], X_{p,\alpha})} \le 1+c_{0,p}\|f_0\|_\alpha,
\]
combined with \eqref{eq2.4a}, \eqref{eq2.4b}, \eqref{eq3.3} and our definition of $\gamma$.
Furthermore, with the aid of \eqref{eq3.4}, \eqref{eq3.5},  in the same manner as above we have for $f,g\in B_r(f^0)$,
\begin{align*}
&\|M(f) - M(g)\|_{C([0,T], X_{p,\alpha})} \\
&\qquad\le
\frac{c_{\alpha,p} T^{1-\alpha}}{1-\alpha} (1+\omega_p+2^{p+3}\kappa)
(1+c_{0,p}\|f_0\|_{p,\alpha})^2 \|f-g\|_{C([0,T], X_{p,\alpha})} \\
&\qquad< \tfrac{1}{2} \|f-g\|_{C([0,T], X_{p,\alpha})}.
\end{align*}
Hence, $M:B_r(f^0)\to B_r(f^0)$ is a contraction and the classical Banach fixed point theorem yields a unique,
mild solution of \eqref{eq1.3} in $B_r(f^0)\subset C([0,T], X_{p,\alpha})$.

(d) To complete the proof we note that the maps $F_1$ and $F_2,$ defined in \eqref{F1F2},
are non-negative in $X_{p,\alpha}^{+}$. Assuming that $f \in B_{r} (f_0)^{+}$, we have
\begin{align*}
[F_2(f)]_i &= \sum_{j=1}^\infty k_{i,j}f_i f_j  \le 2 \kappa \| f \|_{p,\alpha}(1+ \theta_i)^\alpha f_i  \\
&\le 2\kappa(r+\| f_0 \|_{C([0,T],X_{p,\alpha})} )(1+ \theta_i)^\alpha f_i
\leq \gamma(1+\theta_i)^\alpha f_i
\end{align*}
and then $[F_{\gamma,\alpha}(f)]_i \ge 0$, $i\ge 1$.
The last inequality indicates that $B_r (f_0)^{+}$ is invariant under the action of the map $M$
and hence the local mild solution $f$ is non-negative.
\end{proof}

To proceed further, we make use of the following modification of the Gronwall inequality, sometimes called the singular Gronwall inequality, see e.g. \cite[Lemma 8.8.1]{Caz98}. Since wee need some specific aspects of it, we shall provide an elementary proof.
\begin{lemma}\label{lm3.2}
Let $u \in L_{\infty,loc}((0,T]) \cap L_1((0,T))$, $0 < T  < \infty $, be  a nonnegative function satisfying
\begin{equation}
u(t)\leq \frac{c}{t^\gamma} + c_1\int^{t}_0 u(\tau)(t-\tau)^{-\alpha}d\tau, \qquad t\in (0,T],
\label{gron1}
\end{equation}
where $\gamma<1,\, 0 < \alpha < 1$  and $c, c_1> 0$.  Then there is a constant $C(\gamma,\alpha,T),$ independent of $c$,  such that
\begin{equation}
u(t) \leq \frac{cC(\gamma,\alpha,T)}{t^\gamma}, \qquad t\in (0,T].
\label{gron2}
\end{equation}
\end{lemma}
\proof First we observe that, for any $\beta<1, \delta<1$ and $a<b<\infty $, we have
\begin{align}
\int_a^b(b-t)^{-\beta}(t-a)^{-\delta} d t &= (b-a)^{-\beta-\delta+1} \int^{1}_{0}(1-v)^{-\beta}v^{-\delta}d v \nonumber \\&=  B(1-\beta, 1-\delta)(b-a)^{-\beta-\delta+1},
\label{gron3}
\end{align}
where $B$ is the beta function. Since $u$ satisfies \eqref{gron1}, it follows from \eqref{gron3} that
\begin{align}
\int_0^{t}\frac{u(\tau)}{(t - \tau)}^{\alpha} d \tau &\leq c\int_0^{t}\frac{1}{\tau^{\gamma}(t - \tau)^{\alpha}}d \tau  + c_1\int_0^{t}\frac{1}{(t - \tau)^{\alpha}}\left(\int^{\tau}_0 \frac{u(s)}{(\tau - s)^{\alpha}}d s\right) d \tau \nonumber\\
& = c(\theta_\gamma\ast \theta_{\alpha})(t) + c_{1,\alpha}\int_0^{t} u(s) (t - s)^{1-2\alpha} d s,
\label{gron4}
\end{align}
where $\ast$ denotes the Laplace convolution, $\theta_{\kappa}(t) = t^{-\kappa}$ and  $c_{1,\alpha} = c_1B(1-\alpha,1-\alpha)$.  Inserting \eqref{gron4} into \eqref{gron1}, we obtain
\begin{equation}
u(t)\leq c\theta_\gamma(t) + cc_1(\theta_\gamma\ast \theta_{\alpha})(t) + c_{2,\alpha}\int_0^{t} u(\tau) (t-\tau)^{1-2\alpha} d\tau, \qquad t\in (0,T],
\label{gron5}
\end{equation}
with $c_{2,\alpha} = c_1\,c_{1,\alpha}$. Note that the convolution $\theta_\gamma\ast \theta_{\kappa}$ exists for any choice of $\gamma < 1$ and $\kappa < 1$, since
\begin{equation}\label{WLconv}
(\theta_{\gamma}\ast \theta_{\kappa})(t)  = B(1-\gamma,1-\kappa)\,t^{1-\gamma-\kappa} = B(1-\gamma,1-\kappa)\,\theta_{\gamma + \kappa - 1}(t).
\end{equation}
Furthermore,
\begin{equation}\label{WLconv1}
(\theta_{\gamma}\ast \theta_{\kappa})(t)  \leq \frac{\bar C(\gamma, \kappa, T)}{t^\gamma}, \qquad t \in (0,T],
\end{equation}
where $\bar C(\gamma, \kappa, T)$ is a positive constant.

If $1-2\alpha\geq 0$, then we can infer from \eqref{gron5} and \eqref{WLconv1} that
\[
u(t) \leq \frac{c(1 + \bar C(\gamma, \alpha, T))}{t^\gamma} + c_{2,\alpha}t^{1-2\alpha}\int_0^t u(\tau) d \tau,
\]
and then apply the standard arguments used to establish Gronwall-type inequalities  to obtain the desired result. Otherwise, we repeat the above process inductively, using \eqref{WLconv} and \eqref{gron3},  until we arrive at
\begin{equation}
u(t)\leq c\Theta(t) + c^{(k)}_{2^k,\alpha}(u\ast \theta_{2^k\alpha-2^k+1})(t),
\label{gron8}
\end{equation}
where $k \in \mathbb{N}$ is such that $2^k(1-\alpha) - 1 \geq 0$,
\[
\Theta(t) = \theta_\gamma(t) + \sum\limits_{r=1}^{2^k-1} c^{(k)}_{r,\alpha} (\theta_\gamma\ast \theta_{r\alpha-r+1})(t),
\]
and  each $c^{(k)}_{r,\alpha},\ r = 1,2,\ldots k,$ is a positive constant, independent of $c$. Hence,
\begin{equation} \label{WLg1}
u(t)\leq c\Theta(t) + c^{(k)}_{2^k,\alpha} t^{2^{k}(1-\alpha) -1} \int_0^t u(\tau)d \tau.
\end{equation}
Since by $\alpha<1$, we have  $r\alpha-r+1<1$ for any $r\geq 1$, from \eqref{WLconv1} we infer that there is a constant $C_1(\gamma,\alpha,T) > 0$ such that
$$
\Theta(t)\leq \frac{C_1(\gamma,\alpha,T)}{t^\gamma}
$$
and $\Theta$ is integrable on $[0,T]$. Hence a routine argument leads to
$$
u(t) \leq \frac{cC(\gamma,\alpha,T)}{t^\gamma},
$$
for some constant $C(\gamma,\alpha,T)$ and this gives \eqref{gron2}. \hfill\qed

To simplify the notation, in the calculations below, we employ symbol $c$ to denote
a positive constant whose particular value is irrelevant.
\begin{lemma}\label{lm3.3}
Under assumptions of Lemma~\ref{lm3.1}, the mild solution $f$ is H\"older continuous
with exponent $1-\alpha$, i.e.
\begin{equation}\label{eq3.7}
\| f(t+h) - f(t) \|_{p,\alpha} \le \frac{ch^{1-\alpha}}{t^{1-\alpha}},
\end{equation}
for all $t\in (0,T]$ and some $c>0$.
\end{lemma}
\begin{proof}
By virtue of \eqref{eq3.2}, we have
\begin{equation*}
\begin{aligned}
\| f(t+h) - f(t) \|_{p,\alpha} &\le \|S_{\gamma,p,\alpha} (t+h) - S_{\gamma,p,\alpha} (t) ) f_0 \|_{p,\alpha} \\
&+ \Bigl\| \int_{t}^{t+h} S_{\gamma,p,\alpha} (t + h - \tau) F_{\gamma,\alpha} (f(\tau)) d\tau \Bigr\|_{p,\alpha}   \\
&+ \Bigl\| \int_{0}^{t} S_{\gamma,p,\alpha} (\tau) \Bigl(F_{\gamma,\alpha} (f(t+h-\tau))
- F_{\gamma,\alpha} (f(t-\tau))\Bigr) d\tau \Bigr\|_{p,\alpha} \\
&=: J_1 + J_2 + J_3.
\end{aligned}
\end{equation*}
First we infer, by \eqref{eq2.4a} and \eqref{eq2.4c},
\begin{align*}
J_1 &\le \int_0^h \Bigl\|S_{\gamma,p,\alpha} (\tau)\bigl[Y_{\gamma,p,\alpha}
S_{\gamma,p,\alpha} (t)  f_0\bigr]\Bigr\|_{p,\alpha} d\tau
\le c\|S_{\gamma,p,\alpha} (t)  f_0\|_{p,1}\int_0^h \frac{d\tau}{\tau^\alpha}\\
&\le \frac{c h^{1-\alpha}}{t^{1-\alpha}}\|f_0\|_{p,\alpha}
\le \frac{c h^{1-\alpha}}{t^{1-\alpha}}.
\end{align*}
For $J_2$ and $J_3$, in the same manner as in part (c) of Lemma~\ref{lm3.1}, we obtain
\begin{align*}
J_2 &\le  c h^{1-\alpha}\|f\|_{C([0,T], X_{p,\alpha})}^2, \quad J_3 \le c\int_{0}^{t} \|f(\tau+h)-f(\tau)\|_{p,\alpha} \frac{d\tau}{(t-\tau)^\alpha}.
\end{align*}
Combining the estimates, we get
\begin{align*}
\| f(t+h) - f(t) \|_{p,\alpha} \le \frac{ch^{1-\alpha}}{t^{1-\alpha}}
+ c\int_{0}^{t} \|f(\tau+h)-f(\tau)\|_{p,\alpha} \frac{d\tau}{(t-\tau)^\alpha}.
\end{align*}
Hence, the bound \eqref{eq3.7}, with a constant $c>0$ that depends on $\alpha$, $T$ and the initial data $f_0$ only,
follows directly from Lemma~\ref{lm3.2} with $\gamma = 1-\alpha$.
\end{proof}

Lemmas~\ref{lm3.1} and \ref{lm3.3}, combined together, yield
\begin{theorem}\label{th3.4}
Assume that conditions \eqref{eq2.1} and \eqref{eq3.1} are satisfied. Then, for each
$f_0 \in X_{p,\alpha}$ there is  $T = T(f_0) > 0$ such that the initial value problem \eqref{eq1.3} has a unique
non-negative classical solution $f \in C([0, T], X_{p,\alpha}) \cap C^{1}((0, T), X_p) \cap C((0, T), X_{p,1})$.
\end{theorem}
\begin{proof}
First we prove the differentiability of $f$ in $X_p$ for $t>0$. From  \eqref{eq3.2},
\begin{align*}
& \frac{f(t+h) - f(t)}{h}\\
& =   \frac{S_{\gamma,p,\alpha} (h) - I}{h} S_{\gamma,p,\alpha} (t)f_0 + \frac{1}{h} \int_{t}^{t+h} S_{\gamma,p,\alpha} (t + h - \tau) F_{\gamma,\alpha} (f(\tau)) d\tau \\
& + \frac{1}{h} \int_{0}^{t} \bigg (S_{\gamma,p,\alpha} (t + h - \tau)
- S_{\gamma,p,\alpha}(t - \tau) \bigg ) F_{\gamma,\alpha} (f(\tau)) d\tau := I_1 + I_2 + I_3.
\end{align*}
We observe that, by the analyticity, $S_{\gamma,p,\alpha}(t)f_0\in D(T_p)$ for $t>0$, so that
$$\lim_{h\to\infty} I_1 = Y_{\gamma,p,\alpha} S_{\gamma,p,\alpha} (t)f_0$$ in $X_p.$
By \eqref{eq2.4c} we have
\begin{equation}
\| Y_{\gamma,p,\alpha} S_{\gamma,p,\alpha} (t) f_0 \|_{p}
\le  \| S_{\gamma,p,\alpha}(t) f_0 \|_{p,1} \le \frac{c}{t^{1-\alpha}} \| f_0 \|_{p,\alpha}.
\label{SYest}
\end{equation}
The strong continuity of $\{S_{\gamma,p,\alpha}(t)\}_{t\ge 0}$, the continuity of $f$ (see Lemma~\ref{lm3.1}) and
estimates \eqref{eq3.3}, \eqref{eq3.4} combined together, show that in $X_p$
$$\lim_{h\to 0} I_2 = F_{\gamma,\alpha}(f(t)).$$
To find $\lim_{h\to0}I_3$, we first show that
\[
\Bigl\| \int_{0}^{t} Y_{\gamma,p,\alpha} S_{\gamma,p,\alpha} (t - \tau)
F_{\gamma,\alpha}(f(\tau)) d\tau \Bigr\|_{p}  < \infty.
\]
By \eqref{eq3.7}, we have
\begin{align}
&\Bigl\| \int_{0}^{t} Y_{\gamma,p,\alpha} S_{\gamma,p,\alpha} (t - \tau)
F_{\gamma,\alpha}(f(\tau)) d\tau \Bigr\|_{p} \nonumber\\
&\le \int_{0}^{t} \|Y_{\gamma,p} S_{\gamma,p,\alpha} (t - \tau)
( F_{\gamma,\alpha}(f(\tau))- F_{\gamma,\alpha}(f(t)) ) \|_{p} d\tau\nonumber\\
&\qquad+  \Bigl\| \int_{0}^{t} Y_{\gamma,p,\alpha} S_{\gamma,p,\alpha} (t - \tau)
F_{\gamma,\alpha}(f(t)) d\tau\Bigr\|_p\nonumber\\
& \le c \int_{0}^{t} \|S_{\gamma,p,\alpha} (t - \tau)
( F_{\gamma,\alpha}(f(\tau)) - F_{\gamma,\alpha}(f(t)) ) \|_{p,1} d\tau\nonumber\\
&\qquad+  \| (S_{\gamma,p,\alpha} (t) - I) F_{\gamma,\alpha}(f(t))\|_p\nonumber\\
& \le c\int_{0}^{t} \frac{1}{t-\tau}  \|f(\tau) - f(t)\|_{p,\alpha} d\tau+ c \le c t^{\alpha-1}\int_0^t (t-\tau)^{-\alpha}d\tau + c\le c,\label{ests2}
\end{align}
where $c>0$ depends on $t>0$, $\alpha$, $T$, $\kappa$, constants that appear in \eqref{eq2.4} and
the initial data $f_0$. Thus,
$$
\lim_{h\to 0} I_3 = \int_{0}^{t} Y_{\gamma,p,\alpha} S_{\gamma,p,\alpha} (t - \tau)
F_{\gamma,\alpha}(f(\tau)) d\tau.
$$
Combining all our calculations, we conclude that $\frac{df}{dt}\in X_p$ for any $t>0$ and the continuity of each of the above limits shows that $f\in C^1((0,T), X_p)$ is a classical solution.
The same calculations demonstrate also that $f\in C((0,T), X_{p,1})$. \end{proof}

\begin{remark} \label{remB}
The calculations presented above, in particular \eqref{SYest} and the last but one inequality in \eqref{ests2}, show that
$\|Y_{p}f(t)\|_{p} \le \frac{c}{t^{1-\alpha}}$, $t>0,$
provided $f_0 \in X_{p,\alpha}$.
Since the graph norm $\|\cdot\|_p + \|Y_{p}\cdot\|_{p}$ and $\|\cdot\|_{p,1}$ are equivalent in
$X_p \cap D(T_p)$ (as $D(T_p) = D(Y_p)$ and both operators are closed),
it follows that
\begin{equation}
\|f(t)\|_{p,1} \le \frac{c}{t^{1-\alpha}}, \quad t>0,
\label{Ypest}
\end{equation}
 for $f_0\in X_{p,\alpha}$ and hence $\|f\|_{L^1([0,T], X_{p,1})}<\infty$. The last fact is crucial for the numerical analysis presented
in Section~4.
\end{remark}

\subsection{Global non-negative solutions}
Below we show that classical solutions of \eqref{eq1.3} emanating from non-negative initial data are globally defined.
Our analysis requires the following elementary observation.
\begin{lemma}\label{lm3.5}
Assume that $f_0\in X_{p,\alpha}^+$ and for some $\omega_1$
\begin{equation}\label{eq3.8}
\frac{g_i-d_i}{i} -s_i \leq \omega_1
\end{equation}
Then, under the assumptions of Theorem \ref{th3.4}, the local solution satisfies
\begin{equation}\label{eq3.9}
\|f\|_{1} \le e^{\omega_1 t}\|f_0\|_1, \qquad t \in (0, T(f_0)).
\end{equation}
\end{lemma}
\begin{proof}
Since $f\in X_{p,\alpha}^+$, we know that every term of \eqref{eq1.3} is separately well-defined for $t\in (0,T(f_0))$ (as the solution takes values in $D(X_{p,1})$) and differentiable in $X_{p,1},$ and hence in $X_1$. Thus
\[
\frac{d}{dt}\|f(t)\|_1 \le   \sum_{i = 1}^\infty \Bigl(-s_i + \frac{g_i-d_i}{i}\Bigr) i f_i
\le \omega_1\|f(t)\|_1
\]
and \eqref{eq3.9} follows
from the standard Gronwall inequality.
\end{proof}

Two remarks are in place here. First, in the case of pure fragmentation-coagulation models
($s_i=g_i=d_i = 0$, $i\ge 1$) or in the absence of growth ($g_i=0$, $i\ge 1$), we have
$\omega_1 \leq 0$.  Second, even in the absence of sedimentation the bound \eqref{eq3.9} still holds
provided there is a reasonable balance between the growth and the death processes.
\begin{theorem}\label{th3.6}
Under the assumptions of Theorem \ref{th3.4} and  Lemma~\ref{lm3.5}, any solution of \eqref{eq1.3}
with  $f_0\in X_{p,\alpha}^+$, $p>1$, is global in time.
\end{theorem}
\begin{proof}
(a) To begin, we observe that for any $f\in X_{p,\alpha}^+$ we have
\begin{align*}
\sum_{i=1}^{\infty} i^p[Y_p f]_i
& = - \sum_{i=1}^{\infty} i^p\theta_i f_i  \biggl [ \frac{a_i}{\theta_i} \frac{\triangle_i^{(p)}}{i^p}
+ \Bigl(1 - \Bigl(1 - \frac{1}{i}\Bigr)^p \Bigr) \frac{d_i}{\theta_i}  \\
&- \Bigl(\Bigl(1 + \frac{1}{i}\Bigr)^p - 1\Bigr) \frac{g_i}{\theta_i} - \frac{s_i}{\theta_i}\biggr ] \le -c_p \|f\|_{p,1} + \beta_p \|f \|_{p},
\end{align*}
where, by \eqref{eq2.1},  $c_p$ and $\beta_p$ are positive constants that ony depend on the
coefficients of \eqref{eq1.3} and  $p$
(in fact one can take $c_p$ to be any positive constant smaller than $\liminf_{i\to\infty} \frac{a_i}{\theta_i} \frac{\triangle_i^{(p)}}{i^p}$).
By \eqref{eq3.1}, the nonlinearity $F$ admits the bound
\begin{align*}
&\sum_{i=1}^{\infty} i^p F(f)_i
= \frac{1}{2} \sum_{j=1}^{\infty} \sum_{i=1}^{\infty}((i+j)^p - i^p - j^p ) k_{i,j} f_i f_j \\
& \le \frac{2^p - 1}{2} \sum_{j=1}^{\infty} \sum_{i=1}^{\infty}(i^{p-1} j + i j^{p-1} ) k_{i,j} f_i f_j = c_2 (\| f \|_{1} \|f\|_{p-1,\alpha} +
\| f \|_{p-1} \|f\|_{1,\alpha}),
 \end{align*}
with an absolute constant $c_2>0$, where we used the estimate \cite[Eqn. (5.21)]{Bana12b} for the weight.
By \eqref{eq3.10} and the non-negativity of the local classical solution
$f(t)$, for $t\in (0,T)$ we obtain
\begin{subequations}\label{eq3.10}
\begin{equation}\label{eq3.10a}
\frac{d}{dt}\|f\|_p \le -c_p\|f\|_{p,1} + \beta_p\|f\|_p + c_2(\| f \|_{1} \|f\|_{p-1,\alpha}
+ \| f \|_{p-1} \|f\|_{1,\alpha}).
\end{equation}
On the other hand, again by \eqref{eq3.1}, we have
\[
\|F(f)\|_{p} \le (1+2^{p-1}) \sum_{j=1}^{\infty} |f_{j}| \sum_{i=1}^{\infty} i^p k_{i,j}  |f_{i}|
\le 2^{p+1}\kappa \|f\|_{p}\|f\|_{p,\alpha},
\]
while the variation of constant formula and the analiticity of the semigroup $\{S_p(t)\}_{t\ge 0}$
(see estimates \eqref{eq2.4}) imply
\begin{equation}\label{eq3.10b}
\|f(t)\|_{p,\alpha} \le c_{0,p}e^{\omega_p t}\|f_0\|_{p,\alpha} + 2^{p+1}\kappa c_{\alpha,p} \int_0^t
\frac{e^{e^{\omega_p (t-\tau)}}}{(t-\tau)^\alpha}\|f(\tau)\|_p \|f(\tau)\|_{p,\alpha}d\tau.
\end{equation}
\end{subequations}
We use estimates \eqref{eq3.10} to demonstrate that the non-negative local
classical solutions cannot blow up in a finite time. For technical reasons, we separately consider two cases, $1<p\le 2$ and
$2< p<\infty$.

(b) Let $1<p\le 2$. Then \eqref{eq3.10a} implies
\[
\frac{d}{dt}\|f\|_p \le -c_p\|f\|_{p,1} + \beta_p\|f\|_p + 2c_2 \| f \|_{1} \|f\|_{1,\alpha}.
\]
To bound the product term, we use the approach similar to that of \cite{daCo95} and employ H\"older's inequality with the exponent $q = \frac{1}{\alpha}>1$
to obtain
\[
\|f\|_{1,\alpha} \le \|f\|_{p,1}^{\alpha} \|f\|^{1-\alpha}_{\frac{1-p\alpha}{1-\alpha}}
\le  \|f\|_{p,1}^{\alpha} \|f\|^{1-\alpha}_{1}
\]
and then, using Young's inequality,
\[
2c_2\| f \|_{1} \|f\|_{1,\alpha} \le c_p\|f\|_{p,1}
+ (2c_2)^{\frac{1}{1-\alpha}}\Bigr(\frac{\alpha}{c_p}\Bigl)^{\frac{\alpha}{1-\alpha}}
\|f\|_1^{\frac{2-\alpha}{1-\alpha}}.
\]
Hence
\[
\frac{d}{dt}\|f(t)\|_p \le \beta_p\|f\|_p +
(2c_2)^{\frac{1}{1-\alpha}}\Bigr(\frac{\alpha}{c_p}\Bigl)^{\frac{\alpha}{1-\alpha}}
\|f\|_1^{\frac{2-\alpha}{1-\alpha}},
\]
so that the  Gronwall inequality, combined with \eqref{eq3.9}, gives us the bound
\[
\|f(t)\|_p \le \Bigl[1+(2c_2)^{\frac{1}{1-\alpha}}\Bigr(\frac{\alpha}{c_p}\Bigl)^{\frac{\alpha}{1-\alpha}}
\Bigr] e^{\omega'_p t} \|f_0\|_p =:\beta_{\alpha,p} e^{\omega'_p t} \|f_0\|_p,
\]
where $\omega'_p \le \max\Bigl\{\beta_p, \frac{2-\alpha}{1-\alpha}\omega_1\Bigr\}$.
We combine this with \eqref{eq3.10b} to obtain
\[
e^{-\omega_p t}\|f(t)\|_{p,\alpha} \le c_{0,p}\|f_0\|_{p,\alpha} + 2^{p+1}\kappa c_{\alpha,p}\beta_{\alpha,p}
e^{\omega_p' t}\int_0^t \frac{e^{-\omega_p \tau} \|f(\tau)\|_{p,\alpha}}{(t-\tau)^\alpha}d\tau.
\]
Proceeding as in the proof of Lemma~\ref{lm3.2}, we conclude that
\begin{equation}\label{eq3.11}
\|f(t)\|_{p,\alpha} \le C_{p,\alpha}(\|f_0\|_{p,\alpha})e^{\Omega_{p,\alpha} t},
\end{equation}
where $C_{p,\alpha}(\|f_0\|_{p,\alpha})>0$ depends on the coefficients of the model \eqref{eq1.3},
parameter $1<p\le 2$ and the norm $\|f_0\|_{p,\alpha}$ of the initial data,
while the exponent $\Omega_{p,\alpha}>0$ is completely controlled by the parameter $1<p\le 2$ and the
coefficients of \eqref{eq1.3} only. Hence, the case $1<p\le 2$ is settled.

(c) When $2\le p <\infty$, we use H\"older's inequality with the exponent
$q = p' := \frac{p}{p-1}>1$. Since $0<\alpha<1$, we have $\frac{p}{q}(q\alpha-1) \le \alpha$, consequently
\[
\|f\|_{p-1,\alpha} \le \|f\|_{p,1}^{\frac{1}{q}}
\Bigl(\sum_{i =1}^\infty (1+\theta_i)^{ \frac{p}{q}(q\alpha-1)}f_i\Bigr)^{\frac{1}{p}}
\le  \|f\|_{p,1}^{\frac{p-1}{p}}\|f\|_{1,\alpha}^\frac{1}{p}
\]
and, by Young's inequality,
\[
c_2\| f \|_{1} \|f\|_{p-1,\alpha}
\le \frac{c_p}{2}\|f\|_{p,1} +
\Bigl(1-\frac{1}{p}\Bigr)^{1-p} \Bigl(\frac{2c_2}{c_p}\Bigr)^p \|f\|_{1,\alpha}^{p+1}.
\]
Similar procedure yields also
\[
c_2\| f \|_{p-1} \|f\|_{1,\alpha}
\le \frac{c_p}{2}\|f\|_{p,1} +
\Bigl(1-\frac{1}{p}\Bigr)^{1-p} \Bigl(\frac{2c_2}{c_p}\Bigr)^p \|f\|_{1,\alpha}^{p+1}.
\]
Hence, using \eqref{eq3.10a}, we obtain
\[
\frac{d}{dt}\|f\|_p \le \beta_p\|f\|_p + \gamma_p \|f\|_{1,\alpha}^{p+1},
\]
where $\gamma_p>0$ only depends on $p>1$ and the parameters of the model \eqref{eq1.3}.
From part (b)  and the continuity of the embedding $X_{1,\alpha}\subset X_{2,\alpha}$,
 we have
\[
\|f(t)\|_{1,\alpha} \le \|f(t)\|_{2,\alpha} \le C_{2,\alpha}(\|f_0\|_{2,\alpha})e^{\Omega_{2,\alpha} t}
\]
hence $\|f(t)\|_{1,\alpha}$ grows at most exponentially. Hence,
the classical Gronwall inequality yields
\[
\|f(t)\|_p \le \beta_{\alpha,p} e^{\omega'_p t} \|f_0\|_p
\]
also for $2<p<\infty$,  where constants $\beta_{\alpha,p}, \omega'_p>0$ depend on $p$ and the parameters of \eqref{eq1.3} only.
As in part (b) of the proof, the last estimate, together with the inequality \eqref{eq3.10b}, yields the
exponential bound \eqref{eq3.11} for $2<p< \infty$.
We conclude that for any $p>1$, the norm $\|f(t)\|_{p,\alpha}$ of the local solution $f$ emanating from a non-negative initial datum cannot blow-up
in a finite time. Hence, any such solution is defined globally.
\end{proof}

\begin{remark} In the strong sedimentation case, \eqref{eq2.2},  the analysis of Theorems~\ref{th3.4} and ~\ref{th3.6}
extends to the case of $p=1$, since then we also have the analytic fragmentation semigroup  in $X_1$ and the estimates can  be repeated almost verbatim.  In fact, the analysis of Theorem~\ref{th3.6} becomes much simpler as the $X_1$ norm of the solution does not blow up in finite time by Lemma \ref{lm3.5} provided \eqref{eq3.8} is satisfied and thus \eqref{eq3.10b} is immediately applicable with $p=1$.
\end{remark}
\begin{remark} As we mentioned in Introduction, Theorem~\ref{th3.6} significantly extends global solvability results
obtained earlier in the context of pure coagulation-fragmentation model (see \cite{Bana12b}), where
the existence of global solutions is established under much more restrictive assumptions that $\theta_i = a_i \leq c i^s$, $i\ge 1,$ for some constants $c,s >0$  and the exponent $\alpha$ of \eqref{eq3.1} satisfies
 $0< \alpha s\le 1$.
\end{remark}
\section{Numerical Simulations}\label{secNum}
\subsection{The Truncated Problem}
In numerical simulations, we approximate the original infinite dimensional system \eqref{eq1.3}
by the following finite dimensional counterpart:
\begin{equation}\label{eq4.1}
\begin{split}
&\frac{d u_{i}}{d t} = g_{i-1}u_i - \theta_i u_{i} + d_{i+1}u_{i+1}
+ \sum_{j= i+1}^{N} a_{j} b_{i,j} u_{j}\\
&\quad\;\; + \frac{1}{2}\sum_{j=1}^{i-1}k_{i-j,j} u_{i-j}u_{j} -  \sum_{j=1}^{N} k_{i,j} u_{i} u_{j}
+\frac{\delta_{N,i}}{N}\sum_{j=1}^{N}\sum_{n = N+1-j}^N jk_{n,j}u_nu_j,\\
&u_{i}(0) = u_{0,i}, \quad 1 \leq i \leq N.
\end{split}
\end{equation}
The quadratic penalty term ensures that the discrete coagulation process is conservative --
this property is important when dealing with pure coagulation-fragmentation models.

Let  $P_N:X_p\to\mathbb{R}^N$ and  $I_N:\mathbb{R}^N\to X_p$ denote the projector from $X_p$ onto
$\mathbb{R}^N$ and the embedding from $\mathbb{R}^N$ into $X_p$, respectively.
Below, we shall show that if $u^{(N)}$ is the solution of the truncated problem \eqref{eq4.1} with the initial condition $u_0^{(N)}$, then the sequence $I_N u^{(N)}$ approaches $f$ as the truncation index $N$ increases.

\begin{theorem}
\label{th4.1}
Assume \eqref{eq2.1}, \eqref{eq3.1} and \eqref{eq3.8} hold. The truncated problem in
\eqref{eq4.1} is locally solvable, i.e. for each $p>1$ there exists some $T>0$ such that for each $N$
\begin{equation}\label{eq4.2}
u^{(N)} \in C([0, T], X_{p,\alpha}) \cap C^{1}((0, T), X_p) \cap C((0, T), X_{p,1}),
\end{equation}
and the respective norms of $u^{(N)}$ are  bounded independently of $N$.
If, in addition, the initial datum $u^{(N)}_0$ is non-negative,  \eqref{eq4.2} holds
for any fixed $T>0$.
Finally, if for some $q>p-1$, $q\ge 0$ we have $f_0\in X_{q+1,\alpha}^+$ and $\lim_{N\to\infty}
\|I_N u^{(N)}_0 - f_0\|_{p,\alpha} = 0$, then $I_N u^{(N)} \rightarrow f $  in $C([0, T], X_{p,\alpha})$
as $N \rightarrow \infty$.
\end{theorem}
\begin{proof}
(a) System \eqref{eq4.1} is an ODE with a smooth vector field, hence it is locally solvable for any $N>0$.
Let
\begin{align*}
&[Y_N f]_i =  g_{i-1}f_i - \theta_i f_{i} + d_{i+1}f_{i+1}, \quad 1\le i\le N,\quad
[Y_N f]_i = 0, \quad i> N,\\
&[G_N f]_i = \delta_{N+1,i} g_{i-1}f_{i-1},\quad i\ge 1,
\end{align*}
where for each $i \in \mathbb N$, $(\delta_{ij})_{j=1}^\infty$ is the Kronecker delta concentrated at $i$. We see that the linear part of the truncated equation \eqref{eq4.1} acts on the elements of the finite dimensional
subspace $I_N(\mathbb{R}^N) \subset D(T_p)$ according to the formula
\[
Y_N f = Y_p f - G_Nf.
\]
Since the operator $G_N$ is non-negative and bounded, direct application of the variation
of constant  formula implies that the semigroup $\{S_N(t)\}_{t\ge 0}$ generated by
$(Y_N, D(T_p))$ satisfies
\begin{align}
&\|S_{N}(t)\|_{X_p\to X_p}\le \|S_{p}(t)\|_{X_p\to X_p},\quad \|S_{N}(t)\|_{X_p\to X_{p,\alpha}}\le \|S_{p}(t)\|_{X_p\to X_{p,\alpha}}\nonumber\\
&\|S_{N}(t)\|_{X_p\to X_{p,1}}\le \|S_{p}(t)\|_{X_p\to X_{p,1}},
\label{SNest}
\end{align}
so that all estimates involving $\{S_N(t)\}_{t\ge 0}$ are uniform in $N>0$.
Hence, the analysis of Theorems~\ref{th3.4} applies, i.e. for some $T>0$
(that, in general, depends on $p>1$, the initial condition and the coefficients of the problem) inclusion \eqref{eq4.2} holds and the respective
norms are bounded independently of $N$.

Assuming that the initial datum $u^{(N)}_0$ is non-negative, we proceed as in Theorem~\ref{th3.6} to show
that the inclusion \eqref{eq4.2} holds for any fixed $T>0$ uniformly in $N$.
Hence, the first two claims of Theorem~\ref{th4.1} are settled.

(b) To prove the last claim, we derive the equation governing evolution of the numerical error $e^{(N)}(t) := P_N f(t) - u^{(N)}(t)\in \mathbb{R}^N, t\geq 0$. We have
\begin{align*}
&\frac{d e^{(N)}_{i}}{d t}  = g_{i-1}e^{(N)}_{i-1} - \theta_i e^{(N)}_{i} + d_{i+1}e^{(N)}_{i+1}
+\sum_{j= i+1}^{\infty} a_{j} b_{i,j} e^{(N)}_{j}
+ \delta_{N,i} d_{i+1}f_{i+1}\\
&\quad\;\; +  \frac{1}{2}\sum_{j=1}^{i-1}k_{i-j,j}
\bigl ( e^{(N)}_{i-j}f_{j} + u^{(N)}_{i-j} e^{(N)}_{j} \bigr ) -  \sum_{j=1}^{N} k_{i,j} \bigl ( e^{(N)}_{i} f_{j} + u^{(N)}_{i} e^{(N)}_{j} \bigr )\\
&\quad\;\; +  \frac{\delta_{N,i}}{N}\sum_{j=1}^N\sum_{n=N+1-j}^N j k_{j,n} \bigl ( e^{(N)}_{j} f_{n} + e^{(N)}_{n} u^{(N)}_{j}
\bigr )\\
&\quad\;\; -  \frac{\delta_{N,i}}{N}\sum_{j=1}^N\sum_{n=N+1-j}^N j k_{j,n} f_{j}f_{n}
- \sum_{j=N+1}^\infty k_{i,j} f_if_j,\\
&e^{(N)}_{i}(0) = e^{(N)}_{0,i}, \quad 1\le i \le N,
\end{align*}
or, in a compact form,
\[
\frac{d e^{(N)}}{dt} = Y_N e^{(N)} + H_N(t)e^{(N)} + \Bigr(E_N^0 f -  E_N^1 f - E_N^2 f\Bigl), \quad e^{(N)}(0) = e^{(N)}_0,
\]
where, for a given $f$ and $u^{(N)}$, $H_N(t)e^{(N)}$ is linear in $e^{(N)}$ and
\begin{align*}
&[E_N^0 f]_i = \delta_{N,i} d_{i+1} f_{i+1},\quad
[E_N^1 f]_i = \frac{\delta_{N,i}}{N}\sum_{j=1}^N\sum_{n=N+1-j}^N j k_{j,n} f_{j}f_{n},\\
&[E_N^2 f]_i = \sum_{j=N+1}^\infty k_{i,j} f_if_j,\quad 1\le i\le N.
\end{align*}

In what follows we will use two inequalities based on the properties of the function $[0,a]\ni x \mapsto \phi(x):=x^r(a-x)^r$, $a>2, r>0$. Clearly, $\phi$ is symmetric, nonnegative  with $\phi(0)=\phi(a)=0$ and has a single maximum at $x = a/2$. Thus, for $x\in [1,a-1]$ we have $\phi(x) \geq (a-1)^r$. In particular, for $q\geq 0$ and $a=N+1$ we have
\begin{equation}
N^q \leq j^q(N+1-j)^q, \quad 1\leq j\leq N,
\label{Nq}
\end{equation}
where the inequality for $q=0$ is trivial, and for $p\geq 1,$ using \eqref{Nq} and $j\leq N$
\begin{equation}
j^{p-1}(N+1-j)^p = \frac{j^{p}(N+1-j)^p}{j}\geq  N^{p-1}, \quad 1\leq j\leq N.
\label{Np}
\end{equation}
Then, by \eqref{eq3.1},  \eqref{eq3.3}, the fact that $f$ is globally defined by Theorem \ref{th3.6}, and \eqref{Np},
$H_N e^{(N)}$ satisfies
\begin{align*}
\| H_Ne^{(N)} \|_{p} & \le  \frac{1}{2} \sum_{i=1}^{N} i^p \sum_{j=1}^{i-1}k_{i-j,j}
\bigl ( |e^{(N)}_{i-j}| | f_{j}| + |u^{(N)}_{i-j}| | e^{(N)}_{j}| \bigr ) \\
&+  \sum_{i=1}^{N} i^p \sum_{j=1}^{N} k_{i,j} \bigl ( | e^{(N)}_{i}| | f_{j}| + | u^{(N)}_{i} | | e^{(N)}_{j}| \bigr )  \\
&+N^{p-1}\sum_{j=1}^N\sum_{n=N+1-j}^N j k_{j,n} \bigl ( | e^{(N)}_{j}| | f_{n}| + | u^{(N)}_{j} | | e^{(N)}_{n}| \bigr )\\
& \le (1+2^{p+2}) \kappa \|e^{(N)} \|_{p,\alpha} ( \| f \|_{p,\alpha} +  \| u^{(N)} \|_{p,\alpha})  \le \bar c \| e^{(N)} \|_{p,\alpha},
\end{align*}
where \eqref{Np} was used to get
\begin{align*}
&N^{p-1}\sum_{j=1}^N\sum_{n=N+1-j}^N j k_{j,n} \bigl ( | e^{(N)}_{j}| | f_{n}| + | u^{(N)}_{j} | | e^{(N)}_{n}| \bigr )\\
&=\sum_{j=1}^NjN^{p-1}\sum_{n=N+1-j}^N  k_{j,n} \bigl ( | e^{(N)}_{j}| | f_{n}| + | u^{(N)}_{j} | | e^{(N)}_{n}| \bigr )\\
&\leq \sum_{j=1}^Nj^p\sum_{n=N+1-j}^N (N+1-j)^p k_{j,n} \bigl ( | e^{(N)}_{j}| | f_{n}| + | u^{(N)}_{j} | | e^{(N)}_{n}| \bigr ) \\
&\leq \sum_{j=1}^Nj^p\sum_{n=N+1-j}^N n^p k_{j,n} \bigl ( | e^{(N)}_{j}| | f_{n}| + | u^{(N)}_{j} | | e^{(N)}_{n}| \bigr ).
\end{align*}
Similarly, using \eqref{Nq} to estimate $E^1_Nf$,  we have
\begin{align}
&\|E^0_Nf\|_p\le (N+1)\theta_{N+1}|f_{N+1}|,\label{E0}\\
&\|E^1_Nf\|_p\le \bar cN^{p-q-1} \|f\|_{q,\alpha}\|f\|_{q+1,\alpha}\le \bar cN^{p-q-1},\nonumber\\
&\|E^2_Nf\|_p\le \bar c\|(I-P_N)f\|_{p,\alpha},\nonumber
\end{align}
where all generic constants $\bar c>0$ are uniform in $N>0$. The last four bounds, combined
with the variation of constants formula, \[
e^{(N)}(t) = S_{N}(t)e^{(N)}_0 + \int_0^t S_{N}(t-\tau) (H(\tau) e^{(N)}(\tau) + E_N^0 f(\tau) - E_N^1(\tau) - E_N^2(\tau)) d\tau,
\]
\eqref{SNest} and \eqref{eq2.4}, yield
\begin{align*}
\| e^{(N)}(t) \|_{p,\alpha}
&\le \bar c \| e^{(N)}_0 \|_{p,\alpha}
+ \bar c \|(I-P_N)f\|_{C([0,T], X_{p,\alpha})} + \bar cN^{p-q-1}\\
&+ \bar c \int_{0}^{t}  \frac{\| e^{(N)}(\tau) \|_{p,\alpha}}{( t - \tau)^{\alpha}}d\tau + \bar c \int_0^t
\frac{\|E^0_N f(\tau)\|_p}{( t - \tau)^{\alpha}}d\tau,\quad t\in [0,T],
\end{align*}
with a constant $\bar c>0$ that does not depend on the truncation parameter $N>0$. Further, by \eqref{E0} and \eqref{Ypest}, we have
\begin{align*}
\int_0^t
\frac{\|E^0_N f(\tau)\|_p}{( t - \tau)^{\alpha}}d\tau &\leq  \int_0^t
\frac{(N+1)\theta_{N+1}|f_{N+1}(\tau)|}{( t - \tau)^{\alpha}}d\tau
\leq  N^{1-p}\int_0^t \frac{\|f(\tau)\|_{p,1}}{( t - \tau)^{\alpha}}d\tau\\
&\leq \bar{c} N^{1-p}\int_0^t
\tau^{\alpha-1}( t - \tau)^{-\alpha}d\tau = \bar{c} B(\alpha, 1-\alpha) N^{1-p} = \bar{c}N^{1-p},
\end{align*}
where, as before, $\bar{c}>0$ is independent of $N>0$.
Thus, using \eqref{gron2} with $\gamma = 0$ and
\[
c = \bar c (\| e^{(N)}_0 \|_{p,\alpha}
+ \|(I-P_N)f\|_{C([0,T], X_{p,\alpha})} + N^{p-q-1} + N^{1-p})
\]  in a fixed finite time interval
$[0,T]$, we conclude that
\[
\| e^{(N)}(t) \|_{p,\alpha} \le \bar C\Bigl[\| e^{(N)}_0 \|_{p,\alpha} + \|(I-P_N)f\|_{C([0,T], X_{p,\alpha})}
+N^{p-q-1} +  N^{1-p}\Bigr],
\]
with $\bar C>0$ independent of $N>0$.
Note that $\lim_{N\to\infty}\|e^{(N)}_0\|_{p,\alpha} = 0$,
by our assumptions, and the convergence of $\|(I-P_N)f(t)\|_{X_{p,\alpha}}$ to zero is indeed uniform on $[0,T]$ by Dini's theorem.
Hence,
\[
\lim_{N\to\infty} \|I_N u^{(N)} - f\|_{C([0,T],X_{p,\alpha})}  = 0
\]
and the last claim of the theorem is settled.
\end{proof}

\subsection{Simulations}
Below, we provide several numerical illustrations to the theory developed above.
In our simulations, we make use of the following two fragmentation kernels:
\begin{subequations}\label{eq4.3}
\begin{align}
\label{eq4.3a}
&b_{i,j} = \frac{2}{j-1},\\
\label{eq4.3b}
&b_{i,j} = \frac{i^\sigma(j-i)^\sigma}{\alpha_j},\quad \alpha_j
= \frac{1}{j}\sum_{i=1}^{j-1} i^{1+\sigma}(j-i)^\sigma,\quad \sigma>-1.
\end{align}
\end{subequations}
The coagulation process is driven by one of the unbounded kernels (see e.g. \cite{Jace2012, Anki2011,Davi1999}
for the references and particular applications)
\begin{subequations}\label{eq4.4}
\begin{align}
\label{eq4.4a}
&k_{i,j} = k_1 (i^{1/3} + j^{1/3})^{\frac{7}{3}},\\
\label{eq4.4b}
&k_{i,j} = k_2 (i + k_3) (j + k_3),
\end{align}
\end{subequations}
where $k_1$, $k_2$ and $k_3$ are positive constants.
The transport, the sedimentation and the fragmentation rates are chosen to be
\[
g_i = g i^\alpha,\quad d_i = d i^\beta,\quad s_i = s i^\gamma,\quad a_i = a i^\delta,
\]
for all $i\ge 1$, except for $d_1 =a_1 = 0$.

In view of Theorem~\ref{thm2.1}, in the calculations below it is assumed that either
\begin{subequations}\label{eq4.5}
\begin{equation}\label{eq4.5a}
\max\{\alpha,\beta,\gamma\} \le \delta,\quad p>1,
\end{equation}
or
\begin{equation}\label{eq4.5b}
\max\{\beta,\delta\} \le \gamma,\quad p=1,
\end{equation}
\end{subequations}
The conditions ensure that the associated semigroups $\{S_p(t)\}_{t\ge 0}$,
equipped with either of the fragmentation kernels \eqref{eq4.3a} or \eqref{eq4.3b}, are analytic in $X_p$, $p\ge 1$.

\begin{figure}[h!]
\begin{center}
\includegraphics[width=0.45\textwidth]{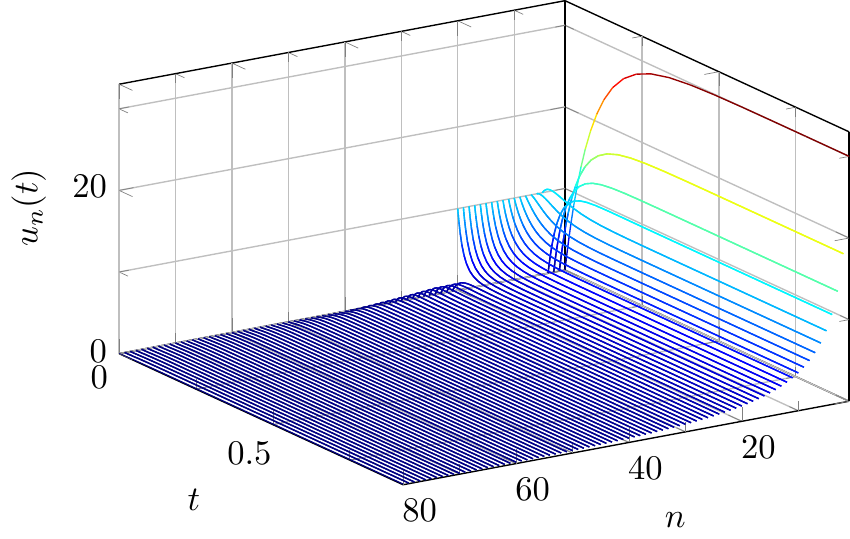}
\hfill
\includegraphics[width=0.45\textwidth]{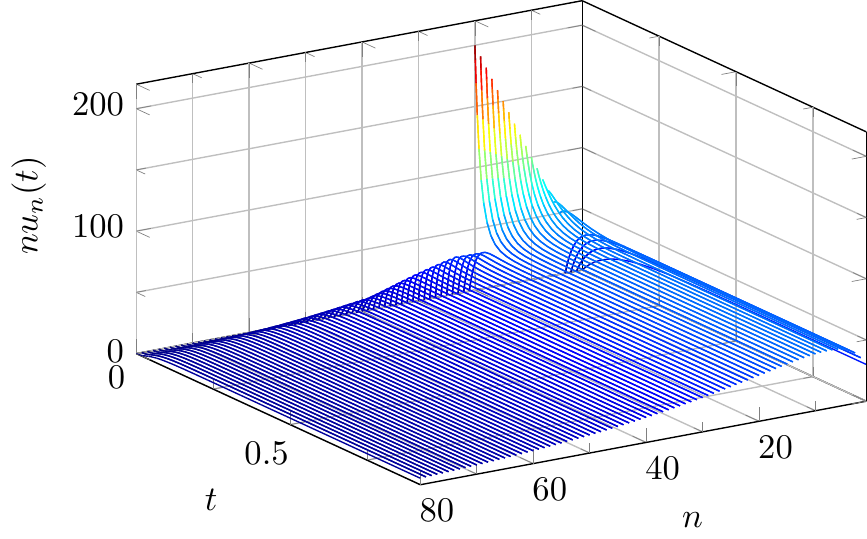}\\
\includegraphics[width=0.45\textwidth]{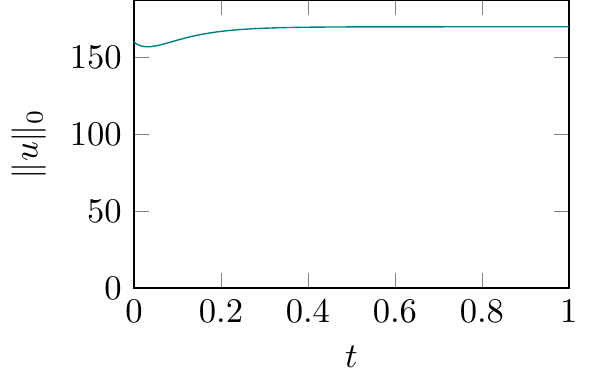}
\hfil
\includegraphics[width=0.45\textwidth]{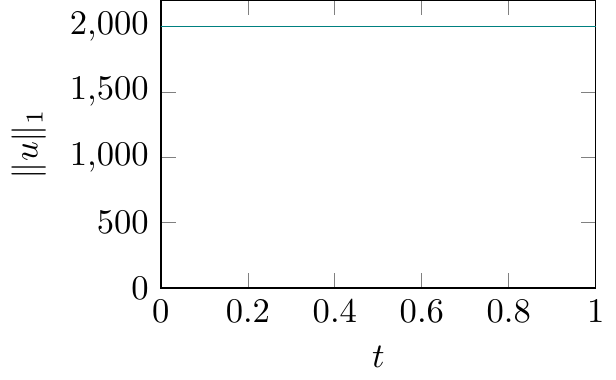}\\
\includegraphics[width=0.45\textwidth]{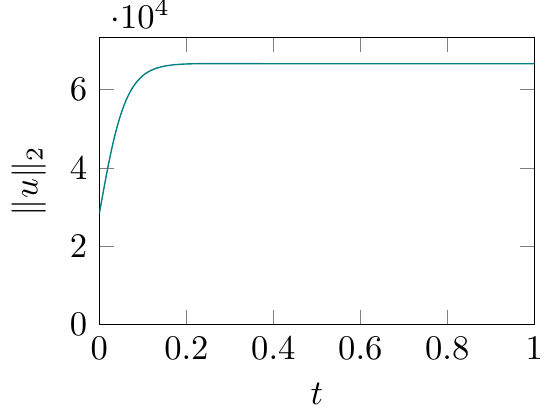}
\hfil
\includegraphics[width=0.45\textwidth]{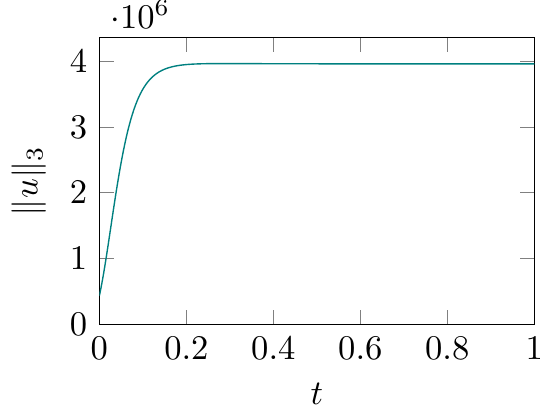}\\
\end{center}
\caption{Evolution of the pure coagulation-fragmentation model \eqref{eq1.3}
with the coagulation kernel \eqref{eq4.4a} and the fragmentation kernel
\eqref{eq4.3a}: number of clusters $u_n(t)$ (top left);
distribution of cluster masses $nu_n(t)$ (top right);
the total number of particles (middle left); the total mass (middle right)
and the higher order moments (bottom). }\label{fig1a}
\end{figure}

\begin{figure}[h!]
\begin{center}
\includegraphics[width=0.45\textwidth]{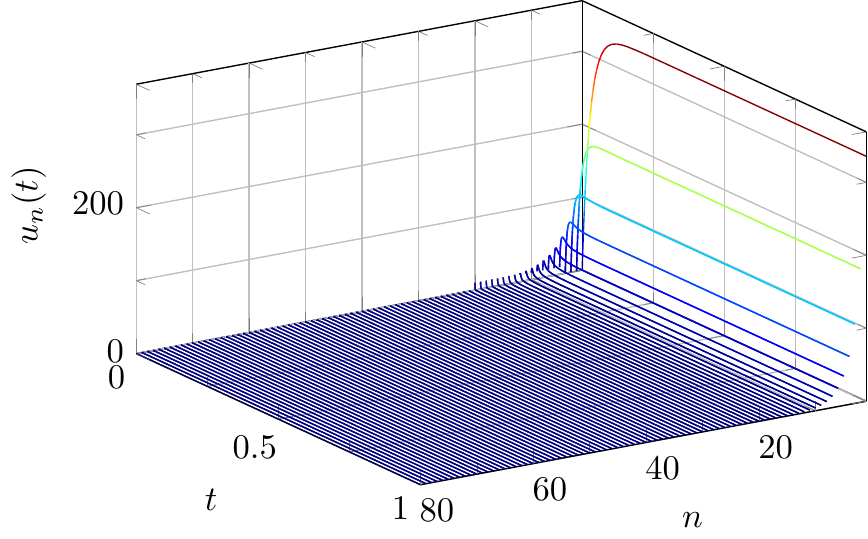}
\hfill
\includegraphics[width=0.45\textwidth]{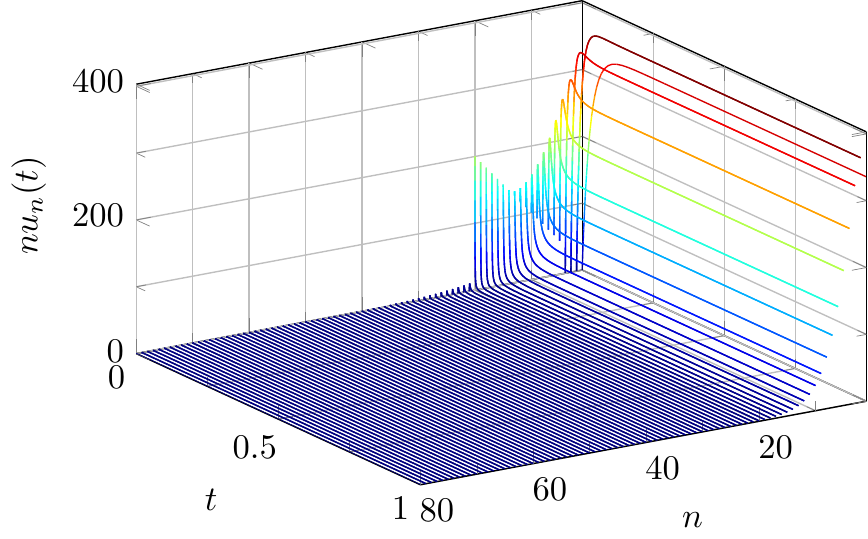}\\
\includegraphics[width=0.45\textwidth]{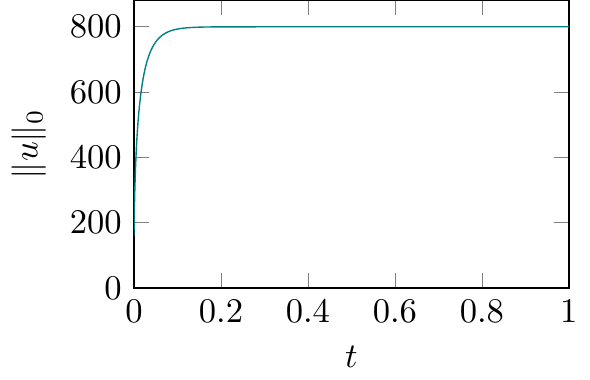}
\hfil
\includegraphics[width=0.45\textwidth]{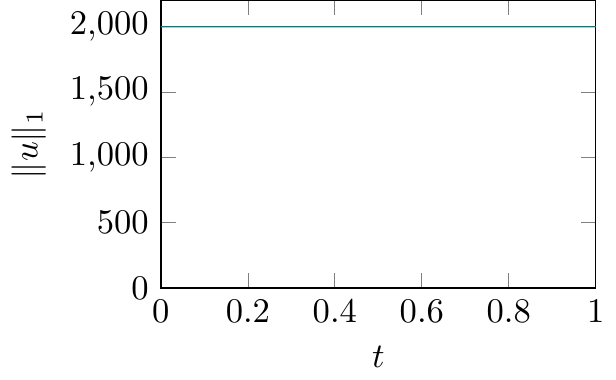}\\
\includegraphics[width=0.45\textwidth]{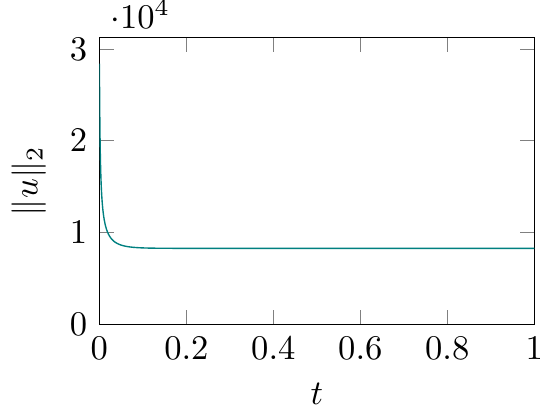}
\hfil
\includegraphics[width=0.45\textwidth]{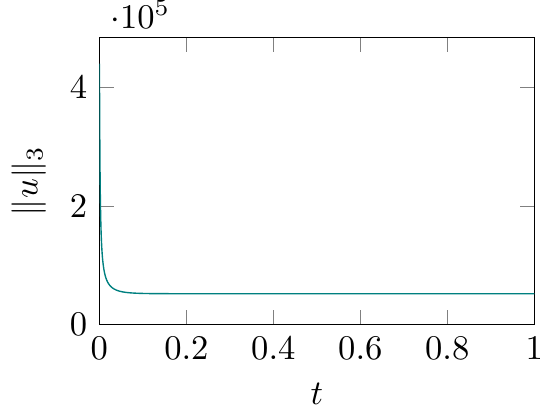}\\
\end{center}
\caption{Evolution of the pure coagulation-fragmentation model \eqref{eq1.3}
with the coagulation kernel \eqref{eq4.4b} and the fragmentation kernel
\eqref{eq4.3b}: number of clusters $u_n(t)$ (top left);
distribution of cluster masses $nu_n(t)$ (top right);
the total number of particles (middle left); the total mass (middle right)
and the higher order moments (bottom). }\label{fig1b}
\end{figure}

\subsubsection{The pure coagulation-fragmentation scenario}
\paragraph{Example 1.} To begin, we consider \eqref{eq1.3} with $g=d=s=0$, fragmentation
kernel \eqref{eq4.3a} and coagulation kernel \eqref{eq4.4a}. Here, the coagulation coefficients
satisfy $k_{i,j} = \mathcal{O}(i^{\frac{7}{9}}+j^{\frac{7}{9}})$ hence Threorem~\ref{th3.6}
applies, provided $\delta>\frac{7}{9}$. In our simulations, we let: $N=200$, $a = 1$,
$\delta = 1$ and $k_1 = 5\cdot 10^{-3}$. Since $N$ is fixed, we shorten the notation setting $u^{(N)} = u$. As the initial conditions, we take
\[
u_n(0) = 10,\quad 5\le n\le 20\quad \text{and}\quad u_n(0) = 0\quad\text{otherwise}
\]
and integrate \eqref{eq4.1} in time interval $[0, 1]$ using \verb"ode15s" built-in \verb"Matlab" ODE solver.
The results of simulations are shown in Fig.~\ref{fig1a}.

At the initial stage (the top left diagram in Fig.~\ref{fig1a}), the coagulation process does generate large
clusters with $n > 20$. However, due to the fragmentation the densities associated with very large particles
steadily go to zero and the solution settles near a steady state distribution. The evolution is further
illustrated by the top right diagram, where the evolution of mass $nu_n(t)$ concentrated
at the clusters of size $1\le n\le 80$ is plotted. As predicted by Theorem~\ref{th3.6}, the strong fragmentation
processes acting in the model prevents uncontrollable mass absorption by the clusters of extremely large sizes.
One can clearly see that after a short transition stage the mass distribution
(concentrated initially in the aggregates of size $5\le n\le 20$) quickly settles near a fixed state, in which
the bulk mass of the ensemble accumulates in clusters of moderate size.

Behaviour of the total number of particles $\|u\|_0$, the total mass of the system $\|u\|_1$ and the
higher order moments $\|u\|_2$, $\|u\|_3$ are shown in the middle and the bottom
diagrams of Fig.~\ref{fig1a}. The middle right diagram shows, in particular, that the process is conservative
(the total mass of the ensemble does not change), while the remaining three diagrams indicate that
the solution settles near a steady state.

\begin{figure}[h!]
\begin{center}
\includegraphics[width=0.45\textwidth]{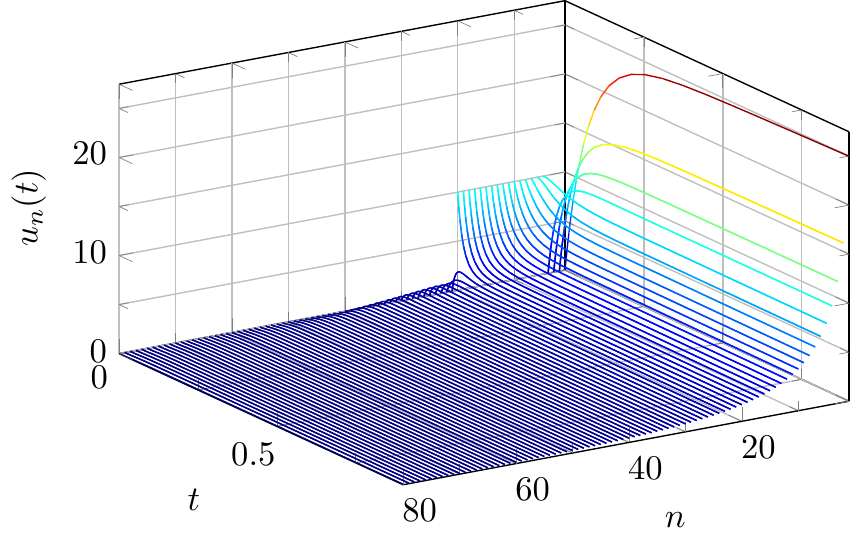}
\hfill
\includegraphics[width=0.45\textwidth]{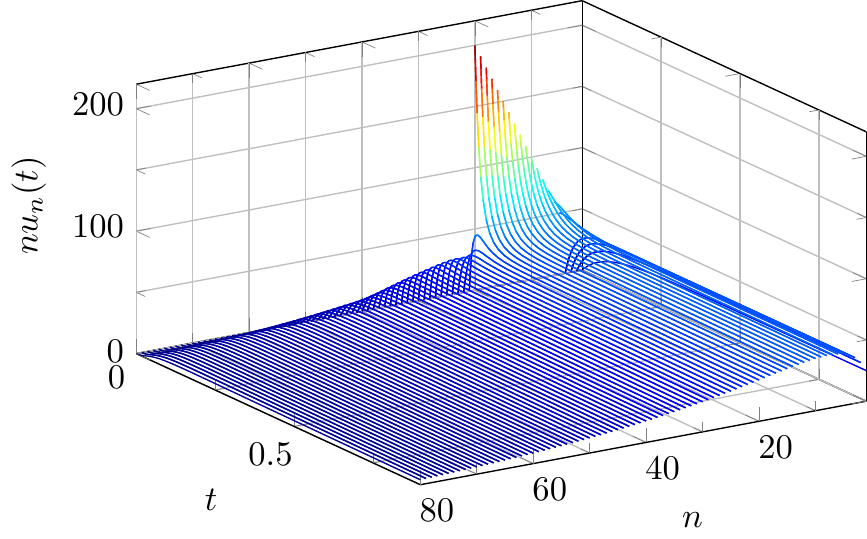}\\
\includegraphics[width=0.45\textwidth]{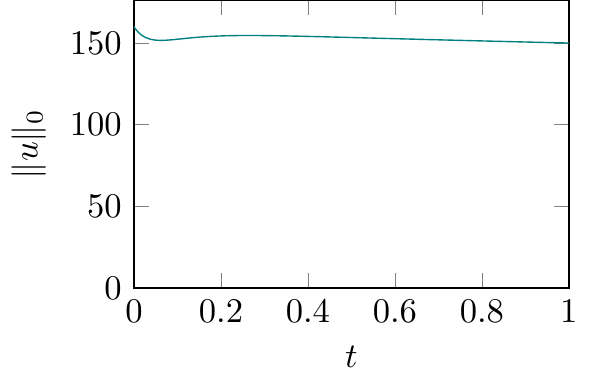}
\hfil
\includegraphics[width=0.45\textwidth]{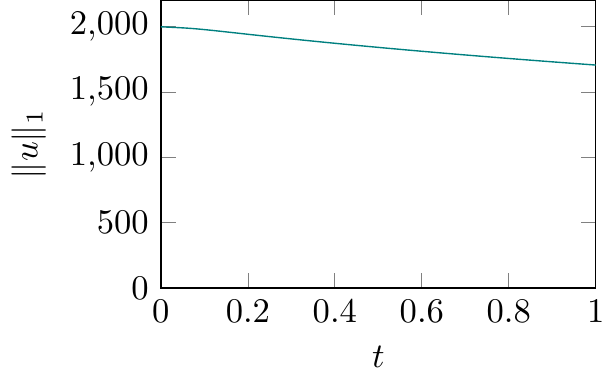}\\
\includegraphics[width=0.45\textwidth]{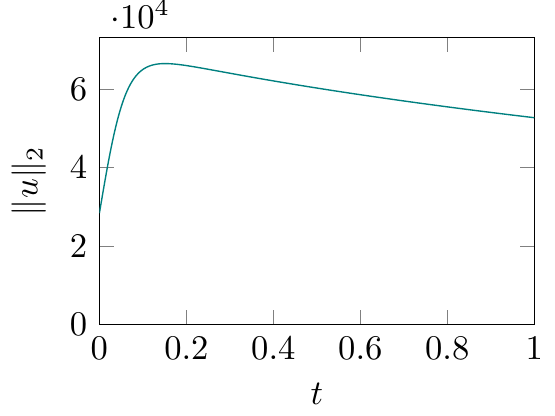}
\hfil
\includegraphics[width=0.45\textwidth]{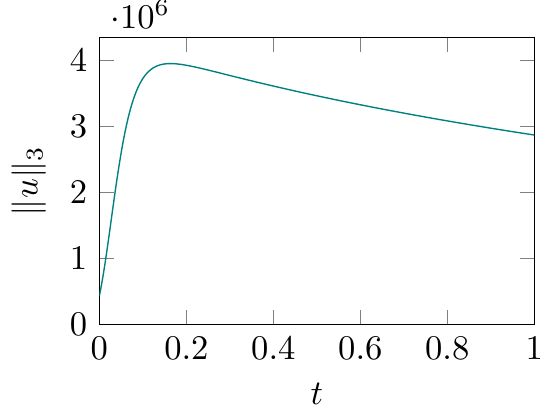}\\
\end{center}
\caption{Evolution of the growth-decay-coagulation-fragmentation model \eqref{eq1.3}
with the coagulation kernel \eqref{eq4.4a} and the fragmentation kernel
\eqref{eq4.3a}: number of clusters $u_n(t)$ (top left);
distribution of cluster masses $nu_n(t)$ (top right);
the total number of particles (middle left); the total mass (middle right)
and the higher order moments (bottom). }\label{fig2a}
\end{figure}

\begin{figure}[h!]
\begin{center}
\includegraphics[width=0.45\textwidth]{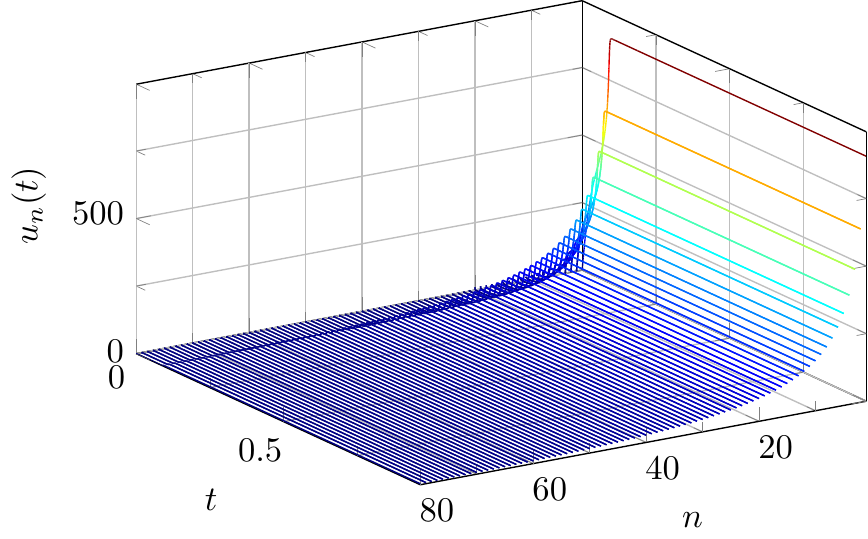}
\hfill
\includegraphics[width=0.45\textwidth]{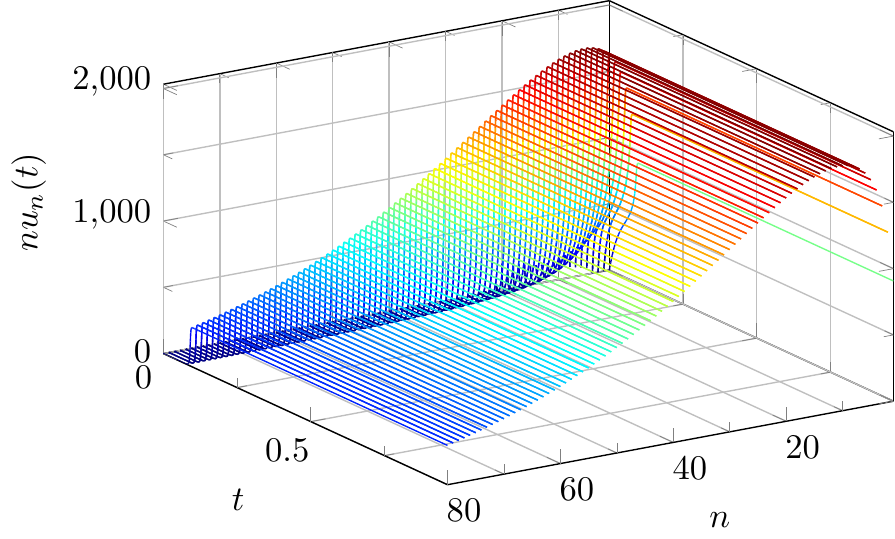}\\
\includegraphics[width=0.45\textwidth]{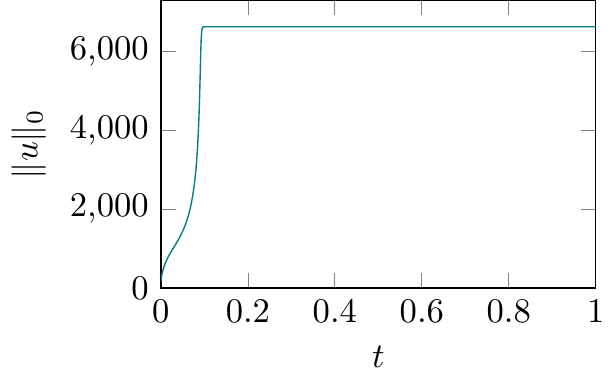}
\hfil
\includegraphics[width=0.45\textwidth]{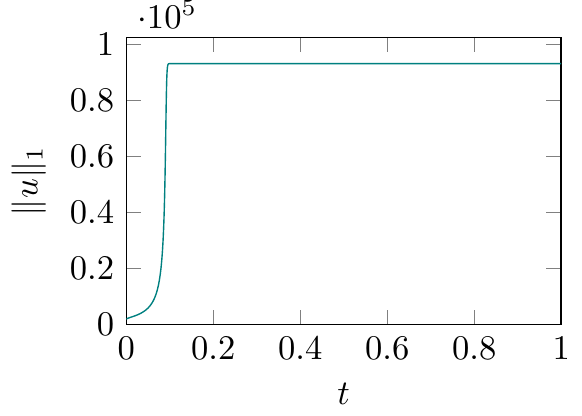}\\
\includegraphics[width=0.45\textwidth]{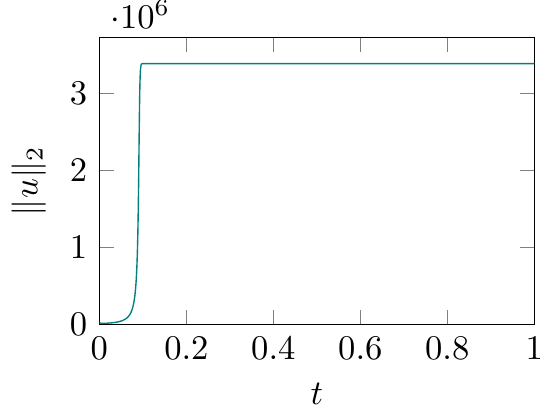}
\hfil
\includegraphics[width=0.45\textwidth]{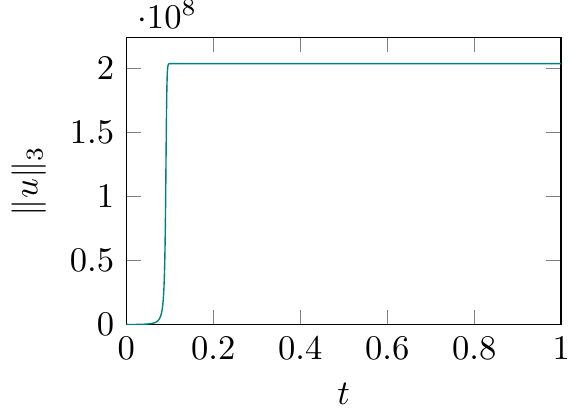}\\
\end{center}
\caption{Evolution of the growth-decay-coagulation-fragmentation model \eqref{eq1.3}
with the coagulation kernel \eqref{eq4.4b} and the fragmentation kernel
\eqref{eq4.3b}: number of clusters $u_n(t)$ (top left);
distribution of cluster masses $nu_n(t)$ (top right);
the total number of particles (middle left); the total mass (middle right)
and the higher order moments (bottom). }\label{fig2b}
\end{figure}

\paragraph{Example 2.} In our second example, we employ the fragmentation kernel \eqref{eq4.3b} with
$\sigma=10^{-1}$ and the coagulation kernel \eqref{eq4.4b} with $k_2 = 5\cdot 10^{-3}$ and $k_3=1$.
Note that $k_{i,j} = \mathcal{O}(i^2+j^2)$ and, in view of \eqref{eq3.1}, we let $\delta=2.5$.
The remaining set of parameters is identical to those used in Example 1.

In the settings described above, the growth rate of the quantities $k_{i,j}$ is superlinear. Hence,
the pure coagulation models lead to a formation of a massive particle outside the system
(the so called gelation phenomenon, see \cite{Wattis2006} and references therein).
In addition, the moment conditions, proposed in \cite{Bana12b} in context of the discrete
pure coagulation-fragmentation models, are also not satisfied. Nevertheless, the example
fells in the scope of Theorem~\ref{th3.6} and, as predicted by the theory, the numerical
solution demonstrates qualitative features similar to those observed in Example 1,
see Fig.~\ref{fig1b}. The total mass is preserved (i.e. no shattering and/or gelation occur) and after a short transition
stage the numerical trajectory settles near a stationary particles/mass distribution.

\begin{figure}[h!]
\begin{center}
\includegraphics[width=0.45\textwidth]{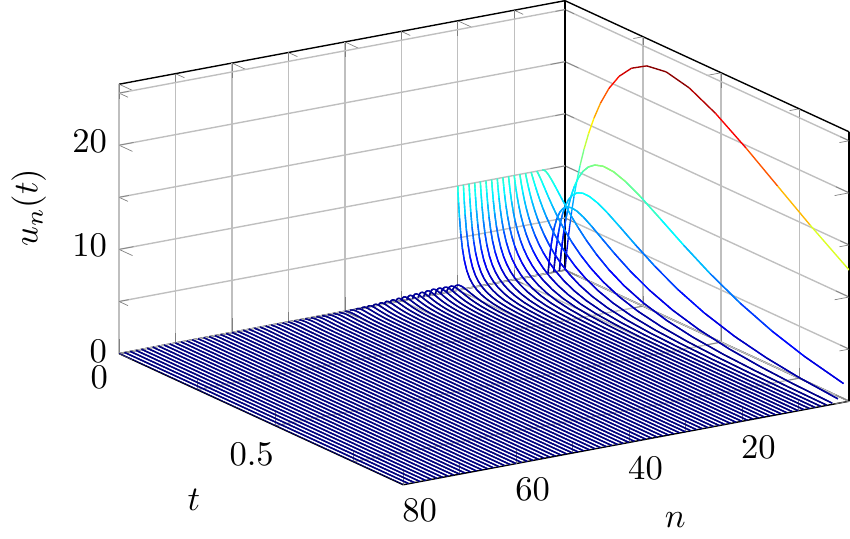}
\hfill
\includegraphics[width=0.45\textwidth]{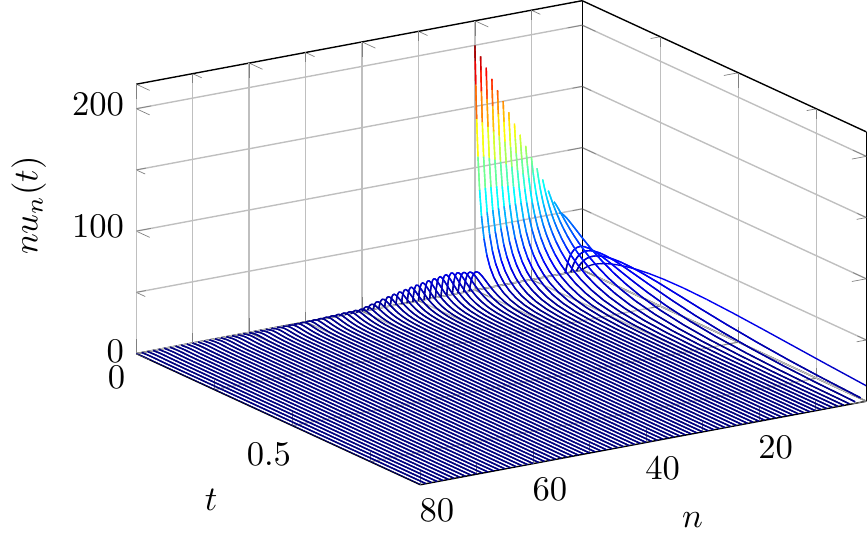}\\
\includegraphics[width=0.45\textwidth]{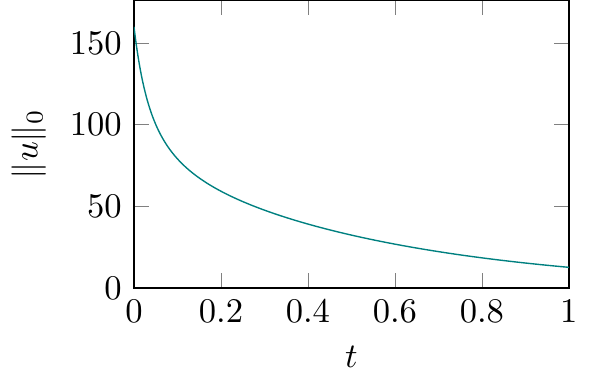}
\hfil
\includegraphics[width=0.45\textwidth]{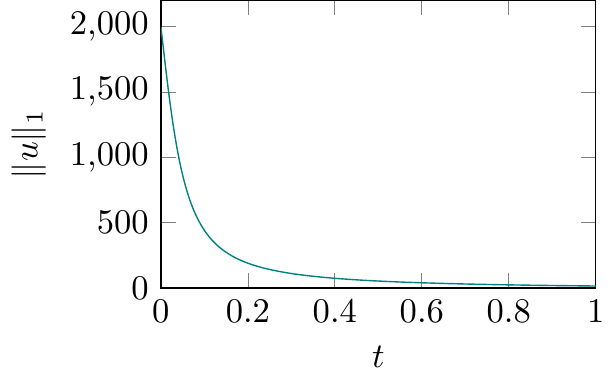}
\end{center}
\caption{Evolution of the decay-sedimentation-coagulation-fragmentation model \eqref{eq1.3}
with the coagulation kernel \eqref{eq4.4a} and the fragmentation kernel
\eqref{eq4.3a}: number of clusters $u_n(t)$ (top left);
distribution of cluster masses $nu_n(t)$ (top right);
the total number of particles (bottom left) and the total mass (bottom right). }\label{fig3a}
\end{figure}

\begin{figure}[h!]
\begin{center}
\includegraphics[width=0.45\textwidth]{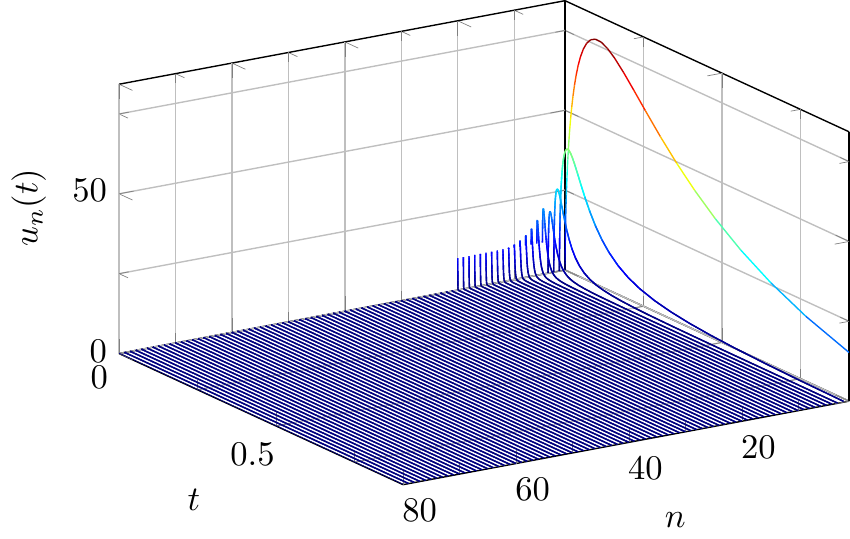}
\hfill
\includegraphics[width=0.45\textwidth]{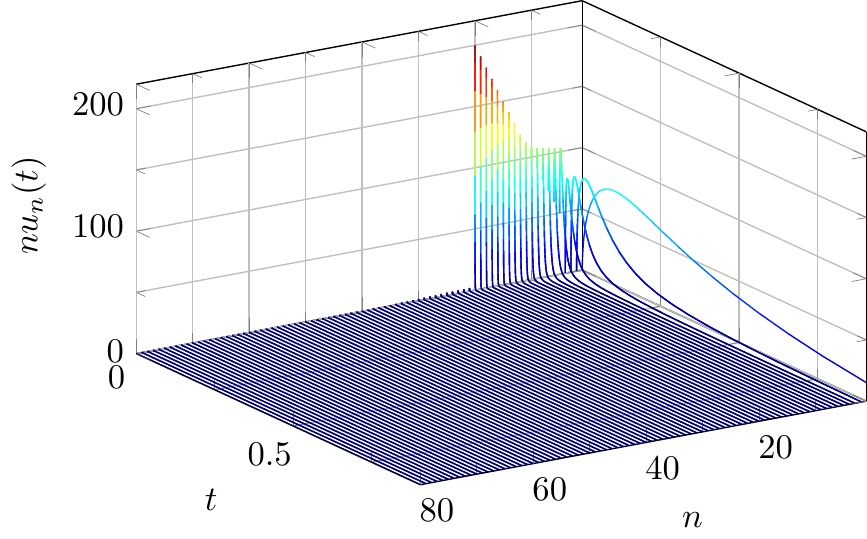}\\
\includegraphics[width=0.45\textwidth]{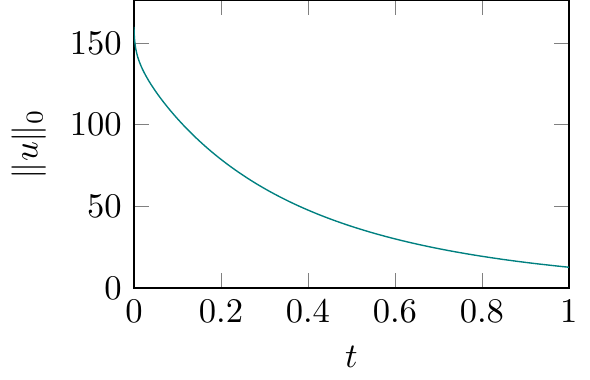}
\hfil
\includegraphics[width=0.45\textwidth]{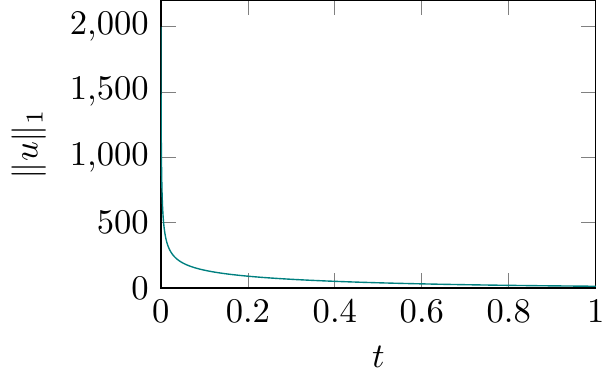}
\end{center}
\caption{Evolution of the decay-sedimentation-coagulation-fragmentation model \eqref{eq1.3}
with the coagulation kernel \eqref{eq4.4b} and the fragmentation kernel
\eqref{eq4.3b}: number of clusters $u_n(t)$ (top left);
distribution of cluster masses $nu_n(t)$ (top right);
the total number of particles (bottom left) and the total mass (bottom right). }\label{fig3b}
\end{figure}

\subsubsection{The growth-decay-sedimentation-fragmentation-coagulation scenario}
\paragraph{Example 3.} We consider the complete model \eqref{eq1.3},
with $g=d=s=a=1$, $\beta=\gamma=0$ and $\alpha=\delta=1$. The fragmentation and the
coagulation processes are controlled respectively by the kernels \eqref{eq4.3a} and \eqref{eq4.4a}, with
$k_1 = 5\cdot10^{-3}$. The truncation index $N$, the time interval $[0, T]$ and the initial condition $u_0$ are
chosen to be the same as in Examples 1 and 2.

As demonstrated by Fig.~\ref{fig2a}, in the presence of the transport processes the qualitative dynamics
of the model \eqref{eq1.3} changes (compare Fig.~\ref{fig2a} with Fig.~\ref{fig1a} and \ref{fig1b}).
The death and the sedimentation processes dominate and yield a slow decay in each of the moments
$\|u\|_p$, $p=0,1,2,3$ as time increases.

\paragraph{Example 4.} To provide a further illustration of the effect of transport processes on the
dynamics of \eqref{eq1.3}, we repeat the computations but with the fragmentation and the coagulation
kernels from Example 2. To ensure global solvability of the model, we let $g=d=s=a=1$, $\beta=\gamma=0$
and $\alpha=\delta=2.5$.

With this settings, the birth and the fragmentation terms dominate and we expect the total mass of the
ensemble to grow. As shown in Fig.~\ref{fig2b}, this is indeed the case for $t$ close to zero. However,
as time goes on, the contributions of the growth and the decay/sedimentation processes compensate each other and
the numerical solution settles near an equilibrium state.

The example demonstrates certain degree of flexibility of model \eqref{eq1.3}. A proper interplay between
the fragmentation and the transport components of the equation allows for simulation of a wide range
of realistic scenarios arising within coupled transport-fragmentation-coagulation systems.

\subsubsection{The no-growth scenario}
Our last two examples demonstrate behaviour of \eqref{eq1.3} in the absence of growth, i.e. when $g=0$ and with sufficently strong sedimentation.
In this settings, the model is globally well posed in $X_1$, provided \eqref{eq2.2} and \eqref{eq3.1} are satisfied.

\paragraph{Example 5.}
We let $g=0$, $d=s=a=1$, $\gamma=\delta=1$ and $\beta = 0$. The fragmentation and
the coagulation kernels and all other parameters are the same as in Example 1.

The results of simulations are shown in Fig.~\ref{fig3a}.
The strong sedimentation (see condition \eqref{eq2.2}) describing  the death
of clusters, prevents uncontrolled mass absorption by the clusters of  large sizes.
The top right diagrams in Fig.~\ref{fig3a} demonstrate that the bulk mass of the system remains
concentrated in clusters of moderate size. As time goes on, both processes lead to a steady decay
in the total mass of the system.

\paragraph{Example 6.} In our last example, we make use of the fragmentation and the coagulation kernels
from Examples 2 and 4. Further, we set $g=0$, $d=s=a=1$, $\gamma=\delta=2.5$ and $\beta = 0$.

As mention earlier, the growth rate of the quantities $k_{i,j}$ is superlinear and
one expects gelation in context of pure coagulation models.
Nevertheless, in complete agreement with the theory, the simulations show (see the evolution of the clusters masses in the top right diagrams
of Fig.~\ref{fig3b})
that in the presence of a sufficiently strong decay-sedimentation process the latter scenario is impossible, and the solution remains bounded in $X_1$ settings (see the bottom
right diagram in Fig~\ref{fig3b}). It is worth to mention that in this example the mechanism preventing gelation
is connected with the strong sedimentation, in contrast to Examples 2 and 4 where the central role is played
by the strong fragmentation.

\section{Conclusion}
In the paper, we considered the discrete coagulation--fragmentation models with growth, decay and sedimentation.
The analysis presented in Section 3 shows that, irrespective of the coagulation rates,  the model is always globally well posed,
provided the fragmentation (in the case of $p>1$), or the sedimentation (for $p=1$) dominate. This is in contrast to
pure coagulation models, see e.g. \cite{Wattis2006}) but confirms earlier results obtained in a more restricted setting in the discrete, \cite{Bana12b, daCo95}, and continuous, \cite{ELMP03}, cases.
Theoretical conclusions are completely supported by the numerical simulations presented in Section 4.

\bibliographystyle{plainnat}
\bibliography{b2018.bib}

\def\cprime{$'$}
\begin{thebibliography}{29}
\providecommand{\natexlab}[1]{#1}
\providecommand{\url}[1]{\texttt{#1}}
\expandafter\ifx\csname urlstyle\endcsname\relax
  \providecommand{\doi}[1]{doi: #1}\else
  \providecommand{\doi}{doi: \begingroup \urlstyle{rm}\Url}\fi

\bibitem[Ackleh and Fitzpatrick(1997)]{Ackleh1}
A.~S. Ackleh and B.~G. Fitzpatrick.
\newblock Modeling aggregation and growth processes in an algal population
  model: analysis and computations.
\newblock \emph{Journal of Mathematical Biology}, 35\penalty0 (4):\penalty0
  480--502, 1997.

\bibitem[Banasiak(2012)]{Bana12b}
J.~Banasiak.
\newblock Global classical solutions of coagulation-fragmentation equations
  with unbounded coagulation rates.
\newblock \emph{Nonlinear Analysis. Real World Applications}, 13\penalty0
  (1):\penalty0 91--105, 2012.

\bibitem[Banasiak and Lamb(2012)]{Jace2012}
J.~Banasiak and W.~Lamb.
\newblock Analytic fragmentation semigroups and continuous
  coagulation-fragmentation equations with unbounded rates.
\newblock \emph{Journal of Mathematical Analysis and Applications},
  391:\penalty0 312--322, 2012.

\bibitem[Banasiak et~al.(2017)Banasiak, Joel, and Shindin]{Banasiak2017}
J.~Banasiak, L.~O. Joel, and S.~Shindin.
\newblock Analysis and simulations of the discrete fragmentation equation with
  decay.
\newblock \emph{Mathematical Methods in the Applied Sciences}, 2017.
\newblock \doi{10.1002/mma.4666}.
\newblock (in print).

\bibitem[{Banasiak} et~al.(2018){Banasiak}, {Joel}, and
  {Shindin}]{Banasiak2018}
J.~{Banasiak}, L.~O. {Joel}, and S.~{Shindin}.
\newblock {Long term dynamics of the discrete growth-decay-fragmentation
  equation}.
\newblock \emph{ArXiv e-prints, arXiv:1801.06486}, 2018.

\bibitem[Banasiak et~al.(2018)Banasiak, Lamb, and Lauren\c{c}ot]{BLL2018}
J.~Banasiak, W.~Lamb, and P.~Lauren\c{c}ot.
\newblock \emph{Analytic Methods for Coagulation--Fragmentation Models}.
\newblock CRC Press, Boca Raton, 2018.
\newblock (in print).

\bibitem[Becker and D{\"o}ring(1935)]{Becker1935}
R.~Becker and W.~D{\"o}ring.
\newblock Kinetische behandlung der keimbildung in {\"u}bers{\"a}ttigten
  d{\"a}mpfen.
\newblock \emph{Annalen der Physik}, 416\penalty0 (8):\penalty0 719--752, 1935.

\bibitem[Bergh and L{\"o}fstr{\"o}m(1976)]{bergh1976}
J.~Bergh and J.~L{\"o}fstr{\"o}m.
\newblock \emph{Interpolation spaces: an introduction}.
\newblock Springer-Verlag, Berlin-New York, 1976.

\bibitem[Bharucha-Reid(1960)]{BR}
A.~T. Bharucha-Reid.
\newblock \emph{Elements of the theory of {M}arkov processes and their
  applications}.
\newblock McGraw-Hill Series in Probability and Statistics. McGraw-Hill Book
  Co., Inc., New York-Toronto-London, 1960.

\bibitem[Blatz and Tobolsky(1945)]{blatz1945}
P.~J. Blatz and A.~V. Tobolsky.
\newblock Note on the kinetics of systems manifesting simultaneous
  polymerization-depolymerization phenomena.
\newblock \emph{The Journal of Physical Chemistry}, 49\penalty0 (2):\penalty0
  77--80, 1945.

\bibitem[Ca{\~n}izo et~al.(2010)Ca{\~n}izo, Desvillettes, and
  Fellner]{Canizo2010}
J.~A. Ca{\~n}izo, L.~Desvillettes, and K.~Fellner.
\newblock Regularity and mass conservation for discrete
  coagulation--fragmentation equations with diffusion.
\newblock \emph{Annales de L'Institut Henri Poincare (C) Non Linear Analysis},
  27\penalty0 (2):\penalty0 639--654, 2010.

\bibitem[Cazenave and Haraux(1998)]{Caz98}
T.~Cazenave and A.~Haraux.
\newblock \emph{An introduction to semilinear evolution equations}, volume~13.
\newblock Oxford University Press, Oxford, 1998.

\bibitem[Collet(2004)]{Collet2004}
J.-F. Collet.
\newblock Some modelling issues in the theory of fragmentation-coagulation
  systems.
\newblock \emph{Communications in Mathematical Sciences}, 2:\penalty0 35--54,
  2004.

\bibitem[Collet and Poupaud(1996)]{Collet1996}
J.-F. Collet and F.~Poupaud.
\newblock Existence of solutions to coagulation-fragmentation systems with
  diffusion.
\newblock \emph{Transport Theory and Statistical Physics}, 25\penalty0
  (3-5):\penalty0 503--513, 1996.

\bibitem[Da~Costa(1995)]{daCo95}
F.~P. Da~Costa.
\newblock Existence and uniqueness of density conserving solutions to the
  coagulation-fragmentation equations with strong fragmentation.
\newblock \emph{Journal of Mathematical Analysis and Applications},
  192\penalty0 (3):\penalty0 892--914, 1995.

\bibitem[David(1999)]{Davi1999}
J.~A. David.
\newblock Deterministic and stochastic models for coalescence (aggregation and
  coagulation): A review of the mean-field theory for probabilists.
\newblock \emph{Bernoulli}, 5:\penalty0 3--48, 1999.

\bibitem[Engel and Nagel(2000)]{Engel2000}
K.~J. Engel and R.~Nagel.
\newblock \emph{One-parameter semigroups for linear evolution equations}.
\newblock Graduate Texts in Mathematics. Springer-Verlag, New York, 2000.

\bibitem[Escobedo et~al.(2003)Escobedo, Lauren{\c{c}}ot, Mischler, and
  Perthame]{ELMP03}
M.~Escobedo, Ph. Lauren{\c{c}}ot, S.~Mischler, and B.~Perthame.
\newblock Gelation and mass conservation in coagulation-fragmentation models.
\newblock \emph{Journal of Differential Equations}, 195\penalty0 (1):\penalty0
  143--174, 2003.

\bibitem[Giri et~al.(2011)Giri, Kumar, and Warnecke]{Anki2011}
A.~K. Giri, J.~Kumar, and G.~Warnecke.
\newblock The continuous coagulation equation with multiple fragmentation.
\newblock \emph{Journal of Mathematical Analysis and Applications},
  374:\penalty0 71--87, 2011.

\bibitem[Gueron and Levin(1995)]{Gue}
S.~Gueron and S.~A. Levin.
\newblock The dynamics of group formation.
\newblock \emph{Mathematical Biosciences}, 128\penalty0 (1):\penalty0 243--264,
  1995.

\bibitem[Jackson(1990)]{Jack90}
G.~A. Jackson.
\newblock A model of the formation of marine algal flocs by physical
  coagulation processes.
\newblock \emph{Deep Sea Research Part A. Oceanographic Research Papers},
  37\penalty0 (8):\penalty0 1197--1211, 1990.

\bibitem[Mirzaev and Bortz(2018)]{Mir18}
I.~Mirzaev and D.~M. Bortz.
\newblock On the existence of non-trivial steady-state size-distributions for a
  class of flocculation equations.
\newblock \emph{arXiv preprint arXiv:1804.00977}, 2018.

\bibitem[Okubo(1986)]{Oku86}
A.~Okubo.
\newblock Dynamical aspects of animal grouping: swarms, schools, flocks, and
  herds.
\newblock \emph{Advances in Biophysics}, 22:\penalty0 1--94, 1986.

\bibitem[Okubo and Levin(2001)]{Oku}
A.~Okubo and S.~A. Levin.
\newblock \emph{Diffusion and ecological problems: modern perspectives},
  volume~14 of \emph{Interdisciplinary Applied Mathematics}.
\newblock Springer-Verlag, New York, second edition, 2001.

\bibitem[Pazy(1983)]{Pa}
A.~Pazy.
\newblock \emph{Semigroups of linear operators and applications to partial
  differential equations}, volume~44 of \emph{Applied Mathematical Sciences}.
\newblock Springer-Verlag, New York, 1983.

\bibitem[Smoluchowski(1916)]{Smoluchowski1916}
M.~Smoluchowski.
\newblock Drei vortrage uber diffusion, brownsche bewegung und koagulation von
  kolloidteilchen.
\newblock \emph{Zeitschrift f\"{u}r Physik}, 17:\penalty0 557--585, 1916.

\bibitem[Smoluchowski(1917)]{Smoluchowski1917}
M.~Smoluchowski.
\newblock Versuch einer mathematischen theorie der koagulationskinetik
  kolloider l{\"o}sungen.
\newblock \emph{Zeitschrift f\"{u}r Physik}, 92:\penalty0 129 -- 168, 1917.

\bibitem[Wattis(2006)]{Wattis2006}
J.~A.~D. Wattis.
\newblock An introduction to mathematical models of coagulation--fragmentation
  processes: a discrete deterministic mean-field approach.
\newblock \emph{Physica D: Nonlinear Phenomena}, 222\penalty0 (1):\penalty0
  1--20, 2006.

\bibitem[Wrzosek(1997)]{Wrzo97}
D.~Wrzosek.
\newblock Existence of solutions for the discrete coagulation-fragmentation
  model with diffusion.
\newblock \emph{Topological Methods in Nonlinear Analysis}, 9\penalty0
  (2):\penalty0 279--296, 1997.

\end{thebibliography}

\end{document}